\documentclass[11pt,a4j,twoside]{article}

\usepackage{amsmath,amssymb}
\usepackage{amsthm}
\usepackage[abbrev]{amsrefs}
\usepackage{latexsym}
\usepackage{graphicx}

\usepackage[english]{babel}
\usepackage{fancyhdr}
\usepackage{a4wide}

\usepackage{bm}

\allowdisplaybreaks[1]

\numberwithin{equation}{section}

\newtheorem{thm}{Theorem}[section]
\newtheorem{prop}[thm]{Proposition}
\newtheorem{lem}[thm]{Lemma}

\newtheorem{dfn}[thm]{Definition}
\theoremstyle{definition} 
\newtheorem{ex}[thm]{Example}
\newtheorem{rem}[thm]{Remark}

\newcommand\ND{\newcommand}

\ND\lref[1]{Lemma~\ref{#1}}
\ND\tref[1]{Theorem~\ref{#1}}
\ND\pref[1]{Proposition~\ref{#1}}
\ND\sref[1]{Section~\ref{#1}}
\ND\ssref[1]{Subsection~\ref{#1}}
\ND\aref[1]{Appendix~\ref{#1}}
\ND\rref[1]{Remark~\ref{#1}}
\ND\cref[1]{Corollary~\ref{#1}}
\ND\eref[1]{Example~\ref{#1}}
\ND\fref[1]{Fig.\ {#1} }
\ND\lsref[1]{Lemmas~\ref{#1}}
\ND\tsref[1]{Theorems~\ref{#1}}
\ND\dref[1]{Definition~\ref{#1}}
\ND\psref[1]{Propositions~\ref{#1}}
\ND\rsref[1]{Remarks~\ref{#1}}
\ND\sssref[1]{Subsections~\ref{#1}}
\ND\esref[1]{Examples~\ref{#1}}
\ND\asref[1]{Assumption~\ref{#1}}

\newcommand{\ep}{\varepsilon}

\newcommand{\h}{\quad}
\newcommand{\dis}{\displaystyle}
\newcommand{\cl}{c\`adl\`ag\ }
\newcommand{\Prob}{\mathbb{P}}

\title{Horizontal $\Delta$-semimartingales on orthonormal frame bundles}
\date{}
\author{Fumiya Okazaki}
\begin{document}
\maketitle
\footnote{Department of Mechanical Engineering, Graduate School of Engineering, The University of Tokyo, Tokyo, Japan}
\footnote{MSC2020 Subject Classifications: 60H05, 60G44}
\footnote{Email address: fumiya-okazaki@g.ecc.u-tokyo.ac.jp}
\renewcommand{\thepage}{\arabic{page}}
\begin{abstract}
In this article, we deal with stochastic horizontal lifts and anti-developments of semimartingales with jumps on complete and connected Riemannian manifolds without any assumption for their curvatures. We prove two one-to-one correspondences among some classes of discontinuous semimartingales on Riemannian manifolds, orthonormal frame bundles and Euclidean spaces by using the stochastic differential geometry with jumps introduced by Cohen (1996). Both of these two results are extension of the one shown in Pontier-Estrade (1992). The first result is the correspondence in the case where jumps of semimartingales are regarded as initial velocities of geodesics which are not necessarily minimal. In the second result, we also established the correspondence in the situation where jumps of semimartingales are given by connection rules, but we impose the condition that the jumps of semimartingales are small. The latter result enables us to construct martingales for a given connection rule with small jumps on any compact manifold from local martingales on a Euclidean space through horizontal semimartingales on orthonormal semimartingales.  
\end{abstract}

\section{Introduction and main theorems}\label{intro}
A stochastic parallel displacement of a frame along a diffusion was defined in \cites{IkeWata, Mal}. This can be regarded as the horizontal lift of a diffusion on a manifold to a frame bundle. The horizontal lift of a continuous semimartingale on a manifold to more general principal bundles was considered in \cites{Shi}. The horizontal lift of semimartingales  is an extension of that of smooth curves on manifolds. Moreover, by employing horizontal lifts, we can regard continuous semimartingales on manifolds as developments of continuous semimartingales on tangent spaces above initial values, which are called anti-developments. Horizontal lifts and anti-developments of discontinuous semimartingales are considered in \cite{PE}, which deals with discontinuous semimartingales on manifolds whose jumps can be connected by unique minimal geodesics. The aim of this article is to extend the result in \cite{PE} so that we can construct anti-developments of discontinuous semimartingales in other situations.

Throughout this paper, we always assume that we are given a filtered probability space $(\Omega, \mathcal{F},\{ \mathcal{F}_t\}_{0\leq t\leq \infty}, P)$ and the usual hypotheses for $\{ \mathcal{F}_t\}_{0\leq t\leq \infty}$ hold. A \cl process $X$ valued in a manifold $M$ is called an $M$-valued semimartingale if $f(X)$ is an $\mathbb{R}$-valued semimartingale for all $f\in C^{\infty}(M)$. To begin with, let us recall basic facts about continuous semimartingales on manifolds. It is known that for an $M$-valued continuous semimartingale $X$, we can define the Stratonovich integral of 1-form $\phi$ along $X$ and the quadratic variation of 2-tensor $\psi$. They are denoted by $\displaystyle \int \phi(X)\circ dX$ and $\displaystyle \int \psi(X)\, d[X,X]$, respectively. Furthermore, given a torsion-free connection on $M$, we can define the It\^o integral of 1-form $\phi$ denoted by $\displaystyle \int \phi(X)dX$ and the equation
\[
\int \phi(X)\circ dX=\int \phi(X) \, dX+\frac{1}{2}\int \nabla \phi (X)\, d[X,X]
\]
holds. Let $\pi :\mathcal{O}(M) \to M$ be an orthonormal frame bundle. AIt is known that given an $M$-valued continuous semimartingale $X$, an $\mathcal{O}(M)$-valued $\mathcal{F}_0$-measurable random variable $u_0$ such that $\pi u_0=X_0$ and a connection $\nabla$ on $M$, the $\mathcal{O}(M)$-valued continuous semimartingale $U$ satisfying $U_0=u_0$, $\pi U=X$ and
\[
\int \theta \circ dU=0
\]
for the connection form $\theta$ corresponding to $\nabla$ is uniquely determined. Furthermore, for an $\mathcal{O}(M)$-valued horizontal semimartingale $U$, the stochastic integral of the solder form $\mathfrak{s}$ along $U$ yields a continuous semimartingale on a Euclidean space. This is called the anti-development of $U$ or of $X:=\pi U$. Conversely, we can construct an $\mathcal{O}(M)$-valued horizontal semimartingale from a continuous semimartingale $W$ starting at $0$ on a Euclidean space. In fact, there exists an $\mathcal{O}(M)$-valued continuous semimartingale $U$ satisfying
\[
F(U_t)-F(U_0)=\int_0^t L_kF(U_s)\circ dW^k_s,\ F\in C^{\infty}(\mathcal{O}(M))
\]
by the existence and uniqueness of solutions of stochastic differential equations (SDE's) on manifolds, where $L_k$ is the canonical horizontal vector field on $\mathcal{O}(M)$ and we have used the Einstein summation convention. It immediately follows that the solution is horizontal. Furthermore, by projecting $U$ onto the base space $M$, we obtain a continuous semimartingale on $M$. In \cite{Hsu}, we can find the proof of the one-to-one correspondence using an embedding of the manifold.

Next, in order to state our main theorems, we briefly describe our setting including discontinuous semimartingales on manifolds. Let $(M,g)$ be a complete and connected Riemannian manifold. In \cite{Pic1}, maps from $M\times M$ to $TM$ called connection rules are introduced and the It\^o integral of 1-forms is defined through connection rules. Given a connection rule $\gamma$, we can determine the direction of jumps of a semimartingale $X$ by $\Delta X_s=\gamma (X_{s-},X_s)$. More generally, a $TM\times M$-valued processes called $\Delta$-semimartingales are introduced in \cite{Pic1}, which are pairs of $M$-valued semimartingales $X$ and jump directions $\Delta X_s\in T_{X_{s-}}M$. Let $\phi$ be a $T^*M$-valued  c\`{a}dl\`{a}g process above $X$, that is, $\phi_s\in T_{X_s}^*M$ for $s\geq 0$. Given a $\Delta$-semimartingale $(\Delta X,X)$ and a torsion-free connection $\nabla$, the It\^o integral of $\phi$ along $(\Delta X,X)$ is defined in \cite{Pic1} and denoted by $\displaystyle \int \phi _-\, dX$. In particular, if $\Delta X$ is written as $\Delta X=\gamma (X_-,X)$ for some connection rule $\gamma$, we denote the stochastic integral of $\phi$ along $(\gamma(X_-,X),X)$ by $\dis \int \phi_-\, \gamma dX$. In a similar way, for a $\Delta$-semimartingale $(\Delta X,X)$, the quadratic variation of a $T^*M\otimes T^*M$-valued c\`{a}dl\`{a}g process $\psi$ above $X$ is also defined and denoted by $\displaystyle \int \psi_{s-}\, d[X,X]_s$. Furthermore, we define the Stratonovich integral of 1-form $\alpha$ as
\[
\int \alpha (X)\circ dX=\int \alpha (X_-)\, dX+\frac{1}{2}\int \nabla \alpha (X_-)\, d[X,X]^c,
\]
where $\displaystyle \int \nabla \alpha (X_-)\, d[X,X]^c$ is the continuous part of the quadratic variation of $\nabla \alpha$.

In this article, we show the two one-to-one correspondences between $\Delta$-semimartingales on manifolds, $\Delta$-horizontal semimartingales on orthonormal frame bundles, and semimartingales on Euclidean spaces starting at $0$. We denote by $\theta$ and $\mathfrak{s}$ the connection form with respect to the Levi-Civita connection and the solder form, respectively. The solder form $\mathfrak{s}$ is a 1-form on $\mathcal{O}(M)$ valued in $\mathbb{R}^d$, where $d=\text{dim}M$. Their precise definition is given in \sref{bundlemetric}. Define the Riemannian metric $\tilde g$ on $\mathcal{O}(M)$ by
\[
\tilde g:=\sum_{\alpha=1}^{\frac{d(d-1)}{2}}\theta^{\alpha}\otimes \theta^{\alpha}+\sum_{k=1}^d\mathfrak{s}^k\otimes \mathfrak{s}^k.
\]
Denote the corresponding Levi-Civita connection by $\tilde \nabla$. Fundamental properties of $\tilde g$ are also given in \sref{bundlemetric}. By employing the connection on the orthonormal frame bundle, we can define the It\^o integral of 1-forms and the quadratic variation of 2-tensors on $\mathcal{O}(M)$ with respect to $\tilde \nabla$. To state our results more precisely, we introduce two more definitions.
\begin{dfn}
\begin{itemize}
\item [(1)]
Let $(\Delta U,U)$ be an $\mathcal{O}(M)$-valued semimartingale. $(\Delta U,U)$ is said to be horizontal if
\[
\int_0^t \theta (U_s)\circ dU_s=0
\]
for all $t\geq 0$.
\item [(2)]
Let $(\Delta U,U)$ be a horizontal $\Delta$-semimartingale on $\mathcal{O}(M)$. The anti-development of $(\Delta U,U)$ is defined by
\[
W=\int \mathfrak{s}(U)\circ dU.
\]
\end{itemize}
\end{dfn}
In particular, if $(\Delta U,U)$ is horizontal, then
\[
\Delta \int_0^t \theta (U_s)\circ dU_s=\langle \theta (U_{s-}),\Delta U_s\rangle=0.
\]
Thus $\Delta U_s$ is a horizontal tangent vector.
\begin{dfn}
Let $(\Delta U,U)$ be an $\mathcal{O}(M)$-valued $\Delta$-horizontal semimartingale and $(\Delta X,X)$ an $M$-valued $\Delta$-semimartingale. $(\Delta U,U)$ will be called a horizontal lift of $(\Delta X,X)$ if
\[
\pi U=X,\ \pi_*\Delta U=\Delta X.
\]
\end{dfn}
Note that given an $\mathbb{R}^d$-valued semimartingale $W$ with $W_0=0$ and an $\mathcal{O}(M)$-valued $\mathcal{F}_0$-measurable random variable $U_0$, there exists an $\mathcal{O}(M)$-valued semimartingale $U$ such that for $F\in C^{\infty}(\mathcal{O}(M))$,
\begin{align*}
F(U_t)-F(U_0)&=\int_0^t L_kF(U_{s-})\circ dW^k_s\\
&+\sum_{0<s\leq t} \{ F(U_s)-F(U_{s-})-L_kF(U_{s-})\Delta W^k_s\}
\end{align*}
by \cite{Co1} or \eref{exmarcus}. $U$ is called a development of $W$. In our first theorem below, we consider $\Delta$-semimartingales satisfying
\begin{align}
\exp_{X_{s-}}\Delta X_s=X_s, \label{eq:v}
\end{align}
where $\exp$ is the exponential map with respect to the Levi-Civita connection.
\begin{thm}\label{main}
\begin{itemize}
\item[(1)]Let $W$ an $\mathbb{R}^d$-valued semimartingale with $W_0=0$ and $U$ a development of $W$. Suppose that $U$ does not explode in finite time. If we set $\Delta U_s=\Delta W^k_s L_k(U_{s-})$, then $(\Delta U,U)$ is a $\Delta$-semimartingale satisfying \eqref{eq:v} on $\mathcal{O}(M)$ and it holds that
\[
\int \theta (U_-)\circ dU=\int \theta (U_-)\, dU=0,
\]
\[
\int \mathfrak{s} (U_-)\circ dU=\int \mathfrak{s} (U_-)\, dU=W.
\]
\item[(2)]Let $(\Delta U, U)$ be a horizontal semimartingale satisfying \eqref{eq:v} on $\mathcal{O}(M)$ and $W$ the anti-development of $(\Delta U,U)$. If $(\Delta V, V)$ is the development of $W$ with $V_0=U_0$, then $U=V,\ \Delta U=\Delta V$, $\Prob$-a.s.
\item[(3)]Let $(\Delta U,U)$ be a horizontal $\Delta$-semimartingale satisfying \eqref{eq:v} on $\mathcal{O}(M)$. Set $X:=\pi U$, $\Delta X:=\pi_*\Delta U$. Then $(\Delta X,X)$ is a $\Delta$-semimartingale satisfying \eqref{eq:v} on $M$.
\item[(4)]Let $(\Delta X,X)$ be an $M$-valued $\Delta$-semimartingale and $u_0$ an $\mathcal{O}_{X_0}(M)$-valued $\mathcal{F}_0$-measurable random variable. Then there exists uniquely a horizontal lift $(\Delta U,U)$ of $(\Delta X,X)$ with $U_0=u_0$ satisfying \eqref{eq:v}. Furthermore, let $(\varepsilon_1,\dots ,\varepsilon_d)$ be a standard basis of $\mathbb{R} ^d$ and $(\varepsilon^1,\dots , \varepsilon ^d)$ its dual basis. Then it holds that
\[
W_t^i=\int_0^tU_{s-}\varepsilon ^i\, dX_s,
\]
where $U_{t-}:(\mathbb{R}^d)^* \to T_{X_{t-}}^*M$ is defined by
\[
\langle U_{t-}a,v\rangle =\langle a,U_{t-}^{-1}v\rangle,\ a\in (\mathbb{R}^d)^*,\ v\in T_{X_{t-}}^*M,
\]
and $W$ is the anti-development of $U$.
\end{itemize}
\end{thm}
This result is a generalization of the result of \cite{PE}. In \cite{PE} this kind of a result was shown only in the case where the jumps of semimartingales can be uniquely connected by a minimal geodesic, but our result includes some cases where this assumption is not satisfied. These situations naturally happen if we consider stchastic differential equations on orthonormal frame bundles on general complete Riemannian manifolds. L\'evy processes on Riemannian manifolds constructed in \cite{App95} are typical examples. 

Our next result includes cases of $\Delta$-semimartingales of which jumps are not described by geodesics. These situations are also significant because in our recent studies \cites{Oka23, Oka24}, it was shown that jumps of martingales on Riemannian submanifolds of Euclidean spaces associated with harmonic maps with respect to non-local Dirichlet forms are not described through geodesics but through embedding. The typical examples of these kinds of harmonic maps are fractional harmonic maps introduced in \cites{Lio, Lio2}, which are critical points of the fractional Dirichlet energy. In \tref{main2} below, we denote geodesic balls on $\mathbb{R}^d$, $M$ and $T_xM$ with radius $r>0$ by $B_r(z)$, $B^M_r(x)$ and $B_r^{T_xM}(v)$, respectively, where $z\in \mathbb{R}^d$, $x\in M$ and $v\in T_xM$.
Our second result is stated as follows. 
\begin{thm}\label{main2}
Let $(M,g)$ be a compact Riemannian manifold and $\gamma \in C^{\infty}(M \times M ; TM)$ a connection rule which induces the Levi-Civita connection. Then there exist $\delta_0=\delta _0(M, g, \gamma)$, $\delta = \delta (M, \gamma) >0$ and $h \in C^{\infty}(\mathcal{O}(M) \times B_{\delta}(0) ; B_{\delta_0}(0))$ such that if we extend the map $h$ to a map on $\mathcal{O}(M) \times \mathbb{R}^d$ by setting $0$ on $\mathcal{O}(M)\times B_{\delta}(0)^c$ and set the map $\varphi \colon \mathbb{R}^d \times \mathcal{O}(M) \times \mathbb{R}^d \to \mathcal{O}(M)$ by
\begin{align}\label{coefficient}
\varphi (z, u, w):=\mathrm{Exp}_u \left( h^k(u,w-z) L_k \right),
\end{align}
then $\varphi$ is a constraint coefficient (see \dref{constraint}) and the following hold:
\begin{itemize}
\item[(1)]Let $Z$ be a semimartingale on $\mathbb{R}^d$ with
\begin{align}\label{jumpbdd}
\sup_{0\leq t<\infty} |\Delta Z_t| < \delta, \ \Prob \text{-a.s.}
\end{align}
Then if $U$ is the solution of the SDE
\begin{align}\label{modifiedSDE}
\overset{\triangle}{d}U=\varphi(U,\overset{\triangle}{d}Z)
\end{align}
with an initial value $U_0$ (see \dref{solution} for the notation in \eqref{modifiedSDE}), the process $X=\pi (U)$ satisfies
\begin{align}\label{intconne}
\int \phi_- \, \gamma dX=\int \langle U_-^{-1}\phi_-, L_k(U_-) \rangle \, dZ^k
\end{align}
for any $T^*M$-valued \cl process $\phi$ above $X$.
\item[(2)]Let $X$ be an $M$-valued semimartingale satisfying
\[
X_t\in \left(\gamma_{X_{t-}}|_{B^M_{\delta_0}(X_{t-})}\right)^{-1}\left(B^{T_{X_{t-}}M}_{\delta}(0)\right)\ \text{for all}\ t\geq 0,\ \Prob \text{-a.s.},
\]
where $\gamma_x:=\gamma(x,\cdot)$ for $x\in M$. Then there exists a semimartingale $Z$ on $\mathbb{R}^d$ such that if $U$ is the solution of \eqref{modifiedSDE} with an initial value $U_0$, then $(U_-h(U_-,\Delta Z),U)$ is a horizontal lift of $X$ and $X$ satisfies \eqref{intconne}.
\end{itemize}
\end{thm}
The precise construction of $\delta_0$, $\delta$ and $h$ appearing in \tref{main2} will be given in the proof. In the case where $M$ is a sphere, we can write them explicitly. See \eref{spmar}. In particular, by taking $Z$ in such a way that $Z$ is a local martingale satisfying \eqref{jumpbdd} in \tref{main2}, we can construct martingales on compact Riemannian manifolds with respect to an arbitrary connection rule which induces Levi-Civita connection from local martingales on Euclidean spaces. As for the precise definition of discontinuous martingales on manifolds, see \dref{defmartingale}.\\

We give an outline of the paper. First we recall connection rules introduced in \cite{Pic1} and define the stochastic integral of 1-forms and the quadratic variation of 2-tensors along $\Delta$-semimartingales on manifolds in \sref{stochasticintegral}. We also recall stochastic differential geometry with jumps developed in \cites{Co1, Co2}. We give proofs of our main results in \sref{sectionmain}. We summarize some facts and simple calculation regarding orthonormal frame bundles and Riemannian metrics on them in \sref{bundlemetric}.
\section{Preliminaries on stochastic integrals and stochastic differential equations}\label{stochasticintegral}
To start with, we recall connection rules introduced in \cite{Pic1}. A connection rule can determine the direction of jumps on a manifold, which is necessary for the definition of the stochastic integral. First, we let $M$ be a $d$-dimensional $C^{\infty}$ manifold.
\begin{dfn}
A mapping $\gamma:M\times M\to TM$ is a connection rule if it is measurable, $C^2$ on a neighborhood of the diagonal set of $M \times M$, and if it satisfies, for all $x,y\in M$,
\begin{enumerate}
\item[(i)] $\gamma (x,y)\in T_xM;$
\item[(ii)] $\gamma (x,x)=0;$
\item[(iii)] $d \gamma (x,\cdot)_x=\mathrm{id}_{T_xM}.$
\end{enumerate}
\end{dfn}
If $(M,g)$ is a Riemannian manifold, we denote by $C_g$ the set of connection rules $\gamma$ which satisfy the following: for all $x,y\in M$, $\exp_x t\gamma (x,y),\ t\in [0,1]$ is a minimal geodesic connecting $x$ and $y$. If $M$ is a strongly convex Riemannian manifold, $\gamma \in C_g$ can be written as
\[
\gamma(x,y)=\exp_x^{-1}y.
\]
In general, as we can observe it in the \pref{connection} below, if $(M,g)$ is a complete connected Riemannian manifold, we can take a connection rule $\gamma$ such that $\gamma (x,y)$ is an initial velocity of a minimal geodesic connecting $x$ and $y$ for all $x,y\in M$ even though the cut locus is not empty. We use the notion $\| \cdot\|$ as the norm with respect to the Riemannian metric.
\begin{prop}\label{connection}
Let $(M,g)$ be a complete and connected Riemannian manifold. Then $C_g \neq \varnothing$. 
\end{prop}
\begin{proof}
Let $\pi :UM \to M$ be a unit tangent bundle. Define $t:UM\to [0,\infty]$ by
\[
t(u):=\sup \{ t\geq 0 \mid d(\pi u, \exp (tu))=t \},
\]
where $d$ is the Riemannian distance. Set
\[
D_x:=\{tu \in TM \mid u\in U_xM,\ t\in [0,t(u)]\},
\]
\[
D:=\bigsqcup _{x\in M}D_x. 
\]
Then $D$ is a closed subset of $TM$. Define $F:TM\to M\times M$ by
\[
F(u):=(\pi u, \exp(u)).
\]
Then $F$ is a continuous map because the solution of the geodesic equations depends continuously on the initial value. We further set
\[
\Phi(x,y):=\{ u\in T_xM \mid \exp _xu=y,\ \| u\| =d(x,y)\}.
\]
Then for a compact subset $C \subset TM$, we have
\begin{align*}
\Phi^{-1}(C)&:=\{(x,y)\mid \Phi(x,y) \cap C \neq \varnothing \}\\
&=F(C\cap D).
\end{align*}
Since $F$ is continuous, this is a compact subset in $M\times M$. In particular, $\Phi ^{-1}(C)$ is a Borel subset. Therefore by measurable section theorem (\cite{Par}, Theorem 5.2), there exists a map $\gamma:M\times M \to TM$ such that $\gamma (x,y)\in \Phi (x,y)$ and $\gamma^{-1}(C)$ is a Borel set for any compact subset $C\subset TM$. Then $\gamma$ is Borel measurable and this is a connection rule we want.
\end{proof}

\begin{rem}
As mentioned in \cite{Pic1}, for each connection rule $\gamma$, there exists a unique torsion-free connection $\nabla$ such that
\[
f(y)-f(x)=\langle df(x), \gamma(x,y) \rangle +\frac{1}{2}\nabla d f(x)(\gamma (x,y),\gamma (x,y))+ \mathrm{o}(d(x,y)^3)\ (y\to x)
\]
for all $f\in C^{\infty}(M)$, where $d$ is a distance compatible with the topology of $M$. Note that the correspondence is not one-to-one. 
For two connection rules $\gamma_1$ and $\gamma_2$, they induce the same connection $\nabla$ if and only if
\begin{align}\label{equiconne}
|\gamma_1(x,y)-\gamma_2(x,y)|=\mathrm{o}(d(x,y)^3)\ (y\to x)
\end{align}
for all $x\in M$.
\end{rem}
Next we recall the definition of $\Delta$-semimartingale introduced in \cite{Pic1}, which is a pair of a c\`{a}dl\`{a}g semimartingale and directions of jumps.
\begin{dfn}\label{Pic3.1}
Let $Y=(\Delta X, X)$ be an adapted $TM\times M$-valued process. The process $Y$ is called a $\Delta$-semimartingale if it satisfies the following:
\begin{itemize}
\item[(i)] $X$ is an $M$-valued semimartingale;
\item[(ii)] $\Delta X_s\in T_{X_{s-}}M$ for all $s>0$;
\item[(iii)] $\Delta X_0\in T_{X_0}M,\ \Delta X_0=0$;
\item[(iv)] for all connection rules $\gamma$ and $T^*M$-valued c\`{a}dl\`{a}g processes $\phi$,
\begin{gather}
\sum_{0<s\leq t} \langle \phi _{s-}, \Delta X_s-\gamma(X_{s-},X_s)\rangle<\infty,\ \text{for all }t>0.\label{iv}
\end{gather}
\end{itemize}
\end{dfn}
\begin{rem}
It is sufficient that condition (iv) is satisfied for some connection rule.
\end{rem}
Let $(\Delta X,X)$ be a $\Delta$-semimartingale and $\nabla$ a torsion-free connection. Then by Proposition 3.5 of \cite{Pic1}, we can define the It\^o integral along $(\Delta X,X)$ by
\[
\int _0^t\phi_{s-} \, dX_s:=\int_0^t \phi_{s-} \, \gamma dX_s+\sum_{0<s\leq t} \langle \phi_{s-}, \Delta X_s-\gamma(X_{s-},X_s) \rangle,
\]
where $\phi$ is a $T^*M$-valued \cl process above $X$ and $\gamma$ is a connection rule which induces $\nabla$. In a similar way, we can define the quadratic variation of 2-tensor along a $\Delta$-semimartingale. For the sake of later use, we introduce discontinuous martingales on manifolds.
\begin{dfn}\label{defmartingale}
Let $M$ be a $d$-dimensional manifold with a torsion-free connection $\nabla$, and $(\Delta X, X)$ an $M$-valued $\Delta$-semimartingale. We call $(\Delta X,X)$ a $\nabla$-martingale if for all $T^*M$-valued \cl processes $\phi$ above $X$, $\dis \int \phi_- \, dX$ is a local martingale.
\end{dfn}
Note that the definition of martingales with jumps depends on the direction of jumps $\Delta X$. If $(\gamma(X_-,X),X)$ is a $\nabla$-martingale, we call $X$ a $\gamma$-martingale.\\

Next we introduce the Stratonovich integral of 1-form along the $\Delta$-semimartingale. The definition below is an extension of that of \cite{PE}.
\begin{dfn}\label{Stratonovich}
Let $\nabla$ be a torsion-free connection and $(\Delta X,X)$ an $M$-valued $\Delta$-semimartingale. For $\alpha \in \Omega^1(M)$, we define
\[
\int_0^t \alpha \circ dX:=\int_0^t \alpha (X_{s-})\, dX_s+\frac{1}{2}\int_0^t (\nabla \alpha)(X_{s-})\, d[X,X]^c_s.
\]
This is called the Stratonovich integral of 1-form along $(\Delta X,X)$. We also denote the integral by $\displaystyle \int_0^t \alpha(X_-)\circ dX.$
\end{dfn}
\begin{prop}\label{falpha}
For $\alpha \in \Omega^1(M)$ and $f\in C^{\infty}(M)$, 
\begin{gather}
\int f \alpha \circ dX=\int f(X)\circ d\left(\int \alpha \circ dX\right).\label{intfa}
\end{gather}
\end{prop}
In \eqref{intfa}, the right-hand side is the Stratonovich integral of $\mathbb{R}$-valued semimartingale $f(X)$ along $\mathbb{R}$-valued semimartingale $\displaystyle \int \alpha \circ dX$; namely, for $\mathbb{R}$-valued semimartingales $Y$ and $Z$,
\[
\int Y\circ dZ=\int Y_-\, dZ+\frac{1}{2}[Y,Z]^c.
\]
\begin{proof}
We begin with the left-hand side of \eqref{intfa}:
\begin{align*}
\int _0^t (f\alpha) (X_s) \circ dX_s=&\int _0^tf\alpha (X_{s-})\ dX_s+\frac{1}{2}\int_0^t (\nabla f\alpha)(X_{s-})\ d[X,X]^c_s\\
=&\int_0^t f(X_{s-})\ d\left( \int_0^{\cdot} \alpha \, dX \right)+\frac{1}{2}\int_0^tf(X_{s-})\ d\left(\int_0^{\cdot} (\nabla \alpha )\ d[X,X]^c\right)\\
&+\frac{1}{2}\int_0^t \alpha \otimes df \ d[X,X]^c_s.
\end{align*}
On the other hand,
\begin{align*}
\int_0^t f(X_{s}&)\circ d\left(\int_0^{\cdot}\alpha \circ dX \right)\\
=&\int_0^t f(X_{s-})\ d\left( \int_0^{\cdot}\alpha \circ dX\right)+\frac{1}{2}\int_0^td \left[ f(X),\int_0^{\cdot}\alpha \ dX \right]^c_s\\
=&\int_0^t f(X_{s-})\ d\left( \int_0^{\cdot}\alpha \ dX\right)+\frac{1}{2}\int_0^tf(X_{s-})\ d\left(\int_0^{\cdot}\nabla \alpha (X_-)\ d[X,X]^c \right)\\
&+\frac{1}{2}\left[ f(X),\int_0^{\cdot}\alpha \ dX \right]^c_t.
\end{align*}
Furthermore,
\begin{align*}
\left[ f(X),\int_0^{\cdot}\alpha \ dX \right]^c_t&=\left[ \int_0^{\cdot}df(X_-)\ dX,\int_0^{\cdot}\alpha \ dX \right]^c_t\\
&=\int_0^tdf\otimes \alpha (X_-)\ d[X,X]^c.
\end{align*}
Therefore we obtain
\[
\int f \alpha \circ dX=\int f(X)\circ d\left(\int \alpha \circ dX\right),
\]
and this is precisely the assertion of the proposition.
\end{proof}

\begin{prop}
Let $(\Delta X, X)$ be a $\Delta$-semimaritngale. Then for $\alpha$, $\beta \in \Omega^1(M)$,
\[
\left[ \int \alpha (X)\circ dX, \int \beta (X)\circ dX \right]=\int \alpha \otimes \beta (X_-)\, d[X,X].
\]
\end{prop}
\begin{proof}
By \dref{Stratonovich} and the property of the Ito integral, we have
\begin{align*}
\left[ \int \alpha (X)\circ dX, \int \beta (X)\circ dX \right]=&\left[ \int \alpha(X_-)\, dX+\frac{1}{2}\nabla \alpha (X_-)\, d[X,X]^c,\right. \\ &\left. \int \beta (X_-)\, dX+\frac{1}{2}\int \nabla \beta (X_-)\, d[X,X]^c \right]\\
=&\left[ \int \alpha (X_-)\, dX,\int \beta(X_-)\, dX \right]\\
=&\int \alpha \otimes \beta (X_-)\, d[X,X].
\end{align*}
This is our claim.
\end{proof}
Since the stochastic integral along a $\Delta$-semimartingale has a \cl modification, we can consider the stochastic integral on a random interval.
\begin{dfn}
Let $S$,$T$ be stopping times with $S<T$ and $(\Delta X,X)$ a $\Delta$-semimartingale. For a $T^*M$-valued c\`{a}dl\`{a}g process $\phi$ above $X$, we define
\begin{align*}
\int_{(S,T]}\phi_{s-}\, dX_s&:=\int_0^T\phi_{s-}\, dX_s-\int_0^S\phi_{s-}\, dX_s,\\
\int_{\{ T\}}\phi_{s-}\, dX_s&:=\langle \phi _{T-}, \Delta X_T\rangle,\\
\int_{(S,T)}\phi_{s-}\, dX_s&:=\int_{(S,T]}\phi_{s-}\, dX_s-\int_{\{ T\}}\phi_{s-}\, dX_s.\\
\end{align*}
We define the quadratic variation and the Stratonovich integral on $(S,T]$, $\{ T\}$, $(S,T)$ in the same way.
\end{dfn}
\begin{prop}\label{coordinate}
Let $(\Delta X, X)$ be a $\Delta$-semimartingale and $(U;x^1,\dots ,x^d)$ a local coordinate neighborhood. Let $\alpha \in \Omega^1(M)$ and $b$ a 2-tensor field with
\[
\alpha=\alpha _idx^i,\ b=b_{ij}dx^i\otimes dx^j\ \text{on}\ U.
\]
Let $S$, $T$ be stopping times such that $S<T$ and $X_s\in U$ for $s\in (S,T)$. Then
\begin{align}
\int_{(S,T)}\alpha (X_{s}) \circ dX_s=&\int_{(S,T)}\alpha_i(X_{s}) \circ dX^i_s+\sum_{S<s<T}\langle \alpha(X_{s-}),\Delta X_s-\Delta X^i_s \frac{\partial}{\partial x^i}\rangle, \label{eq:1}\\
\int_{(S,T)}b(X_{s-})\, d[X,X]_s=&\int_{(S,T)}b_{ij}(X_{s-})\, d[X^i,X^j]_s\notag \\ 
+\sum_{S<s<T}\{b(X_{s-})&(\Delta X_s,\Delta X_s)-b(X_{s-})(\Delta X^i_s \frac{\partial}{\partial x^i},\Delta X^i_s \frac{\partial}{\partial x^i})\}, \label{eq:2}
\end{align}
where $X^i=x^i(X)$ on $U$.
\end{prop}
\begin{proof}
We begin with the left-hand side of \eqref{eq:1}:
\begin{align*}
\int_{(S,T)}\alpha (X)\circ dX=&\int_{(S,T)}\alpha_i(X_-)\ d\left(\int dx^idX \right)\\
&+\frac{1}{2}\int_{(S,T)}\left( \frac{\partial \alpha_i}{\partial x^j}-\Gamma ^k_{ij}\alpha _k \right)(X_-)d\left( \int dx^i\otimes dx^j \ d[X,X]^c\right).
\end{align*}
On the other hand,
\begin{align*}
&\int_{(S,T)}\alpha_i(X)\circ d\left( \int dx^i\circ dX \right)\\
=&\int_{(S,T)}\alpha_i(X)\circ d\left(\int dx^i dX \right)+\frac{1}{2}\int _{(S,T)}\alpha_i(X)\circ d\left(\int \nabla dx^i d[X,X]^c \right)\\
=&\int_{(S,T)}\alpha_i(X_-)\ d\left(\int dx^i\ dX \right)+\frac{1}{2}\int_{(S,T)}\left(\frac{\partial \alpha_k}{\partial x^j}-\Gamma ^i_{jk}\alpha_i \right)(X_-)\ dx^j\otimes dx^k\ d[X,X]^c.
\end{align*}
Therefore
\[
\int_{(S,T)}\alpha (X)\circ dX=\int_{(S,T)}\alpha_i(X)\circ d\left( \int dx^i\circ dX\right).
\]
Set $\displaystyle \Delta X_s=a^i_s\left( \frac{\partial}{\partial x^i} \right)_{X_{s-}}$ Then
\begin{align*}
\int_{(S,T)}\alpha_i(X)& \circ d\left( \int dx^i\circ dX \right)\\
&=\int_{(S,T)}\alpha_i(X)\circ d\left(X^i-X^i_0-\sum_{0<s\leq \cdot}(X^i_s-X^i_{s-}-a^i_s) \right)\\
&=\int_{(S,T)}\alpha_i(X)\circ dX^i+\sum_{S<s<T}\langle \alpha (X_{s-}),\Delta X_s-\Delta X^i_s\frac{\partial}{\partial x^i}\rangle.
\end{align*}
\eqref{eq:2} follows in the same way.
\end{proof}
Next, we recall the theory of second-order stochastic differential geometry with jumps. In \cites{Co1, Co2}, Cohen formulated the stochastic integral of order 2 along a c\`{a}dl\`{a}g semimartingale valued in a manifold and the stochastic defferential equation. In this section we summarize results about them. See \cites{Co1, Co2, AVMU, Vecc} for details. Several works done in \cite{Co1, Co2} were also summarized in \cite{Maillard06}.
A linear map $L:C^2(M)\to \mathbb{R}$ is called a second-order differential operator without constant at $x\in M$ if for a local coordinate $(x^i)$ including $x$, there exist $a^{ij}\in \mathbb{R}$, $b^k\in \mathbb{R}$ ($i,j,k=1,\dots,d$) such that $L$ is denoted by
\[
Lf(x)=\sum_{i,j=1}^da^{ij}\frac{\partial^2f}{\partial x^i \partial x^j}(x)+\sum_{k=1}^db^k\frac{\partial f}{\partial x^k}(x),\ f\in C^2(M).
\]
This definition does not depend on local coordinates.
Denote the vector space of all second-order differential operators at $x\in M$ by $\mathbb{T}_xM$. The space
\[
\mathbb{T}M=\bigsqcup_{x\in M}\mathbb{T}_xM
\]
is called the second-order tangent bundle on $M$. Let $\mathbb{T}^*_xM$ be the dual space of $\mathbb{T}_xM$ for each $x\in M$. The space
\[
\mathbb{T}^*M=\bigsqcup_{x\in M}\mathbb{T}^*_xM
\]
is called the second-order cotangent bundle on $M$.
\begin{dfn}
Let $f:M\to \mathbb{R}$ be a Borel measurable function. $f$ is called a form of order 2 specified in $x$ if $f$ is twice differentiable at $x$ and $f(x)=0$. We call $x$ the base point of $f$ and denote it by $\pi (f)$. Define
\begin{align*}
\overset{\triangle}{\mathbb{T}^*}_xM&:=\{f:M\to \mathbb{R}\mid f\text{ is a form of order 2 specified in }x \},\\
\overset{\triangle}{\mathbb{T}^*}M&:=\bigsqcup_{x\in M} \overset{\triangle}{\mathbb{T}^*}_xM,
\end{align*}
and $\mathcal{G}_x:\overset{\triangle}{\mathbb{T}^*}_xM\to \mathbb{T}^*_xM$ by
\[
\mathcal{G}_xf(L):=Lf(x),\ f\in \overset{\triangle}{\mathbb{T}^*}_xM,\ L \in \mathbb{T}M.
\]
\end{dfn}


By Theorem 1 of \cite{Co1}, for an $M$-valued semimartingale $X$ and a $\overset{\triangle}{\mathbb{T}^*}M$-valued predictable locally bounded process $\Theta$ above $X$, we can define the stochastic integral $\dis \int \Theta \, \overset{\triangle}{d}X$. Note that the stochastic integral defined in Theorem 1 of \cite{Co1} can recover the It\^o integral with respect to connection rules.
Let $\gamma$ be a connection rule on $M$, $X$ an $M$-valued semimartingale and $\phi$ a $T^*M$-valued c\`{a}dl\`{a}g process. We set
\begin{gather}
\left( \overset{\triangle}{\gamma}_{X_{s-}}\phi_{s-}\right)(y)=\langle \phi_{s-},\gamma(X_{s-},y)\rangle,\ y\in M.\label{process2}
\end{gather}
Then $\overset{\triangle}{\gamma}_{X_{s-}}\phi_{s-}\in \overset{\triangle}{\mathbb{T}^*}_{X_{s-}}M$ and
\[
\int \overset{\triangle}{\gamma}_{X_{s-}}\phi_{s-}\, \overset{\triangle}{d}X_s=\int \phi_{s-}\, \gamma dX_s.
\]
Next we recall the theory of SDE's on manifolds with jumps formulated in \cites{Co1, Co2}.
\begin{dfn}\label{constraint}
Let $M$ and $N$ be manifolds. Suppose that $C$ is a closed submanifold of $M\times N$ such that the projection $p_1$ from $C$ to $M$ is onto and a submersion. A measurable map $\varphi:C\times M \to N$ is called a constraint coefficient from $C\times M$ to $N$ if
\begin{itemize}
\item[(i)]for each $(x,y)\in C$, $\varphi (x,y,x)=y$,
\item[(ii)]$\varphi$ is $C^3$ in a neighborhood of $\{ (z,\ p_1(z))| z\in C\}$,
\item[(iii)]for all $x\in M$ and $z\in C$, $(x,\ \varphi (z, x))\in C$.
\end{itemize}
\end{dfn}

\begin{dfn}\label{solution}
Let $M$ and $N$ be manifolds, $C$ a closed submanifold of $M \times N$, and $\phi:C\times M \to N$ a constraint coefficient from $C\times M$ to $N$. Fix an $M$-valued semimartingale $X$ and an $N$-valued $\mathcal{F}_0$-measurable random variable $y_0$ with $(X_0,y_0)\in C$. A pair of a positive predictable stopping time $\eta$ and an $N$-valued semimartingale $Y$ on $[0,\eta)$ is called a solution of the SDE
\begin{align}
\left\{ \begin{array}{ll}
\overset{\triangle}{d}Y=\phi(Y,\overset{\triangle}{d}X),\\
Y_0=y_0,
  \end{array} \right.\label{SDE}
\end{align}
if $Y_0=y_0$, $(X,Y)\in C$ and for all $\overset{\triangle}{\mathbb{T}^*}N$-valued locally bounded predictable processes $\Theta$ with $\Theta_t\in \overset{\triangle}{\mathbb{T}^*}_{Y_t}N$ on $[0,\eta)$,
\[
\int \Theta \overset{\triangle}{d}Y=\int \phi^*\Theta \overset{\triangle}{d}X,
\]
where
\[
\phi^*\Theta_t(z):=\Theta_t (\phi (X_{t-}, Y_{t-},z)).
\]
\end{dfn}
Theorem 2 of \cite{Co1} and Theorem 1 of \cite{Co3} guarantee that equation \eqref{SDE} admits a unique solution.
\begin{rem}\label{solution2}
Let $\iota \colon M \to \mathbb{R}^d$ be an embedding. Then by Remark 7 and  Proposition 4 of \cite{Co1}, $(Y, \eta)$ is a solution of the SDE \eqref{SDE} if and only if for any $f\in C^{\infty}(N)$ and $t<\eta$,
\begin{align*}
f(Y_t)-f(Y_0) &= \int_0^t \frac{\partial f \circ \Phi_s}{\partial z^i} (X_{s-})\, dX^i_s+\frac{1}{2}\int_0^t\frac{\partial^2 f\circ \Phi_s}{\partial z^i \partial z^j}(X_{s-})\, d[X^i,X^j]^c_s\\
&\h +\sum_{0<s\leq t}\left\{f(\Phi_s(X_s))-f(Y_{s-})-\frac{\partial f \circ \Phi_s}{\partial z^i} (X_{s-})\Delta X^i_s \right\},
\end{align*}
where $\Phi_t \colon \mathbb{R}^d \to N$ is an extension of $\phi (X_{t-},Y_{t-},\cdot)$ to a function on $\mathbb{R}^d$ which is $C^2$ at $z=X_{t-}$, $(z^1,\dots,z^d)$ is the standard coordinate on $\mathbb{R}^d$ and
\[
\iota (X_t)=(X^1_t,\dots,X_t^d).
\]
\end{rem}
\begin{ex}\label{exmarcus}
Let $A_i$ ($i=1,\dots,r$) be a complete vector field on $M$. We suppose that any $\mathbb{R}$-linear combination of $\{A_i\}_{i=1,\dots,r}$ is also complete. Let $h \colon M \times \mathbb{R}^r\to \mathbb{R}^r$ be a function which is $C^{\infty}$ on the neighborhood of $M \times \{0 \}$ and
\[
h(x,0)=0
\]
for all $x\in M$. Define $\phi \colon \mathbb{R}^r \times M \times \mathbb{R}^r \to M$ by
\[
\phi (z, x, w):= \mathrm{Exp}_x \left( \sum_{k=1}^r h^k(x, w-z) A_k \right),
\]
where $h(x,z)=(h^1(x,z),\dots,h^r(x,z))$. Then $\phi$ is a constraint coefficient.
Let $W$ be a $d$-dimensional semimartingale with $W_0=0$. Then the SDE
\begin{gather}
\overset{\triangle}{d}X=\phi(X,\overset{\triangle}{d}W) \label{SDE2}
\end{gather}
admits a unique solution $(X,\eta)$. By \rref{solution2}, this means that for all $f\in C^{\infty}(M)$, $X$ satisfies that
\begin{align*}
f(X_t)-f(X_0)&=\int_0^t A_kf(X_{s-})\frac{\partial h_s^k}{\partial w^i}(W_{s-}) \, dW^i_s\\
&\h +\frac{1}{2}\int_0^t \left\{ A_kA_lf(X_{s-})\frac{\partial h_s^k}{\partial w^i}(W_{s-})\frac{\partial h_s^l}{\partial w^j}(W_{s-})\right. \\
&\h \left. + A_kf(X_{s-}) \frac{\partial^2 h^k_s}{\partial w^i \partial w^j}(W_{s-}) \right\} \, d[W^i,W^j]^c_s\\
&\h +\sum_{0<s\leq t}\left\{ f(\mathrm{Exp}_{X_{s-}}(h_s^k(W_s)A_k))-f(X_{s-})-A_kf(X_{s-})\frac{\partial h_s^k}{\partial w^i}(W_{s-}) \Delta W^i_s \right\},
\end{align*}
where
\[
h^k_t(w)=h^k(X_{t-}, w-W_{t-}).
\]
\end{ex}
\lref{intvec} below states that we can describe the Stratonovich integral and the quadratic variation along the solution of \eqref{SDE2} by the integral along the driving semimartingale on a Euclidean space.
\begin{lem}\label{intvec}
Set $\Delta X_t:= A_k(X_{t-})h^k(X_{t-},\Delta W_t)$ under the conditions stated in \eref{exmarcus}. Suppose $(\Delta X,X)$ is a $\Delta$-semimartingale and $h$ satisfies
\[
d_0 h(x,\cdot)=\mathrm{id} \colon \mathbb{R}^r \to \mathbb{R}^r
\]
for all $x\in M$, where the left hand side is the derivative of the map $h(x,\cdot)\colon \mathbb{R}^r \to \mathbb{R}^r$ at the origin. Then for $\alpha \in \Omega^1(M)$ and $\beta \in \Gamma(T^*M\otimes T^*M)$,
\begin{align*}
\int \alpha (X)\circ dX&=\int \langle \alpha,A_k\rangle (X)\circ dW^k\\
&\h + \frac{1}{2}\int \frac{\partial^2 h^k_s}{\partial w^i \partial w^j}(W_{s-})\langle \alpha, A_l\rangle \, d[W^k,W^l]^c\\
&\h + \sum_{0<s\leq t}\langle \alpha (X_{s-}), R^k_sA_k(X_{s-}) \rangle, \\
\int \beta (X)\, d[X,X]&=\int \beta (X_-)(A_k(X_-),A_l(X_-))\, d[W^k,W^l]\\
&\h +\sum_{0<s\leq \cdot} \Delta W^k_sR^l_s df \otimes dg (A_k,A_l)(X_{s-}) \nonumber \\
&\h +\sum_{0<s\leq \cdot} R^k_s \Delta W^l_s df \otimes dg (A_k,A_l)(X_{s-}) \nonumber \\
&\h +\sum_{0<s\leq \cdot} R^k_sR^l_s df \otimes dg (A_k,A_l)(X_{s-}),
\end{align*}
where we set
\[
R^k_t:=h^k(X_{t-},\Delta W_t) - \Delta W^i_t.
\]
\end{lem}
\begin{proof}
Without losing generality, we can suppose that $M$ is isometrically embedded in a higher dimensional Euclidean space by $\iota \colon M \to \mathbb{R}^d$. Note that for any $\alpha \in \Omega^1(M)$ and $\beta \in \Gamma (T^*M\otimes T^*M)$ we can write
\begin{align}
\alpha &= \alpha_i d\iota^i,\\
\beta &= \beta_{ij} d\iota^i \otimes d\iota^j,
\end{align}
where $\alpha_i = \langle \alpha, \nabla \iota^i \rangle$, $\beta_{ij}=\beta (\nabla \iota^i, \nabla \iota^j)$.
To begin with, we take $f\in C^{\infty}(M)$ and $\alpha = df$. Then by It\^o's formula,
\[
f(X_t)-f(X_0)=\int_0^t df (X) \circ dX + \sum_{0<s\leq t}\{ f(X_s)-f(X_{s-})-\langle df(X_{s-}) , \Delta X_s\rangle \}
\]
Note that
\[
A_k(X_{s-})\frac{\partial h^k_s}{ \partial w^i}(W_{s-})=A_i(X_{s-}).
\]
by the assumption for $h$. Then we have
\begin{align*}
\int_0^t df (X) \circ dX &= f(X_t)-f(X_0)-\sum_{0<s\leq t}\{ f(X_s)-f(X_{s-})-\langle df(X_{s-}) , \Delta X_s\rangle \} \\
&= \int_0^t A_kf(X_{s-})\frac{\partial h^k_s}{ \partial w^i}(W_{s-}) \, dW^i_s\\
&\h +\frac{1}{2}\int_0^t \left\{ A_kA_lf(X_{s-})\frac{\partial h_s^k}{\partial w^i}(W_{s-})\frac{\partial h_s^l}{\partial w^j}(W_{s-})\right. \\
&\h \left. + A_kf(X_{s-}) \frac{\partial^2 h^k_s}{\partial w^i \partial w^j}(W_{s-}) \right\} \, d[W^i,W^j]^c_s\\
&\h +\sum_{0<s\leq t}\langle df(X_{s-}), \Delta X_s - A_k (X_{s-}) \frac{\partial h^k_t}{ \partial w^k}(W_{s-}) \Delta W^i_s \rangle \\
&= \int_0^t \frac{\partial h^k_s}{ \partial w^i}(W_{s-}) \langle df, A_k \rangle (X_{s-}) \, dW^k_s\\
&\h +\frac{1}{2}\int_0^t \left\{ \frac{\partial h_s^k}{\partial w^i}(W_{s-})\frac{\partial h_s^l}{\partial w^j}(W_{s-}) A_kA_lf(X_{s-}) \right. \\
&\h \left. + \frac{\partial^2 h^k_s}{\partial w^i \partial w^j}(W_{s-}) \langle df, A_k \rangle (X_{s-})  \right\} \, d[W^i,W^j]^c_s\\
&\h +\sum_{0<s\leq t}\left\langle df(X_{s-}), \left(h^k(X_{t-},\Delta W_t) -  \frac{\partial h^k_t}{ \partial w^i}(W_{t-})\Delta W^i_t\right) A_k (X_{s-}) \right\rangle \\
&= \int_0^t \langle df, A_i \rangle (X_{s-}) \, dW^i_s +\frac{1}{2}\int_0^t A_iA_jf(X_{s-})\, d[W^i,W^j]^c_s \\
&\h + \frac{1}{2} \int_0^t \frac{\partial^2 h^k_s}{\partial w^i \partial w^j}(W_{s-}) \langle df, A_k \rangle (X_{s-}) \, d[W^i,W^j]^c_s\\
&\h +\sum_{0<s\leq t}\left\langle df(X_{s-}), R^k_s A_k (X_{s-}) \right\rangle.
\end{align*}
Moreover, by substituting $f$ to $A_if$ in the above equality, we have
\begin{align*}
\int_0^t \langle df, A_i\rangle (X) \circ dW^i_s&=\int_0^t \langle df, A_i\rangle (X_{s-})\, dW^i_s+\frac{1}{2}[A_if(X), W^j]^c_s\\
&=\int_0^t \langle df, A_i\rangle (X_{s-})\, dW^i_s+\frac{1}{2}\int_0^t A_iA_jf(X_{s-})\, d[W^i,W^j]^c_s.
\end{align*}
Thus we have
\begin{align}\label{intvect2}
\int_0^t df (X) \circ dX &=\int_0^t \langle df, A_k\rangle (X) \circ dW^k_s \nonumber \\
&\h + \frac{1}{2} \int_0^t \frac{\partial^2 h^k_s}{\partial w^i \partial w^j}(W_{s-}) \langle df, A_k \rangle (X_{s-}) \, d[W^i,W^j]^c_s \nonumber \\
&\h +\sum_{0<s\leq t}\left\langle df(X_{s-}), R^k_sA_k (X_{s-}) \right\rangle.
\end{align}
Next, we take another $g\in C^{\infty}(M)$. Then by \eqref{intvect2},
\begin{align}
\int df \otimes dg (X)\, d[X,X]&= \left[ \int df(X)  \circ dX, \int dg(X) \circ dX  \right] \nonumber \\
&= \int df \otimes dg (A_k,A_l)(X_-) \, d[W^k, W^l] \nonumber \\
&\h +\sum_{0<s\leq \cdot} \Delta W^k_sR^l_s df \otimes dg (A_k,A_l)(X_{s-}) \nonumber \\
&\h +\sum_{0<s\leq \cdot} R^k_s \Delta W^l_s df \otimes dg (A_k,A_l)(X_{s-}) \nonumber \\
&\h +\sum_{0<s\leq \cdot} R^k_sR^l_s df \otimes dg (A_k,A_l)(X_{s-}).
\end{align}

\end{proof}



\section{Proofs of \tsref{main} and \ref{main2}}\label{sectionmain}
Let $(M,g)$ be a complete and connected Riemannian manifold and $\mathcal{O}(M)$ an orthonormal frame bundle on $M$. We use the notation defined in the previous section. In this section, we prove \tsref{main} and \ref{main2}. Throughout the proof, we will frequently use facts about orthonormal frame bundles and a bundle metric on $\mathcal{O}(M)$. We denote the bundle metric and the associated Levi-Civita connection by $\tilde g$ and $\tilde \nabla$, respectively. See \dref{defmetric} for the detail of them. We summarize the basic setting and some facts in \sref{bundlemetric}.
\subsection{Proofs of \tref{main} (1)--(3)}
\begin{proof}[Proof of \tref{main} (2)]
By the definition of the stochastic development given just above \tref{main}, it holds that for all $F\in C^{\infty}(\mathcal{O}(M))$, 
\begin{align*}
F(V_t)-F(V_0)=&\int_0^t L_kF(V_s)\circ dW^k_s\\
&+\sum_{0<s\leq t} \{ F(V_s)-F(V_{s-})-L_kF(V_{s-})\Delta W^k_s\}.
\end{align*}
On the other hand, by It\^o's formula,
\begin{align}
F(U_t)-F(U_0)=&\int_0^tdF(U_s)\circ dU_s \nonumber \\
&+\sum_{0<s\leq t} \{ F(U_s)-F(U_{s-})-\langle dF(U_{s-}),\Delta U_s\rangle \}.\label{UIto}
\end{align}
Since $(\Delta U,U)$ is horizontal and satisfies $\Delta W^k=\langle \mathfrak{s}^k,\Delta U\rangle$, it holds that
\[
\Delta U_s=\Delta W^k_s L_k(U_{s-})=\Delta V_s,\ s\geq 0.
\]
Note that $dF=L_kF\mathfrak{s}^k+X_{\alpha}^{\sharp} F\theta ^{\alpha}$, where $X^{\sharp}_{\alpha}$ is a vertical vector field on $\mathcal{O}(M)$ given in \eqref{verticalvec} in \sref{bundlemetric}. Then by \pref{falpha}, equation \eqref{UIto} can be written as
\begin{align*}
F(U_t)-F(U_0)=&\int_0^tL_kF(U_s)\circ dW^k_s\\
&+\sum_{0<s\leq t} \{ F(U_s)-F(U_{s-})-L_kF(U_{s-})\Delta W^k_s\}.
\end{align*}
This implies that $U$ is also the stochastic development of $W$. Therefore by uniqueness of the solution of the SDE, we have $U=V$, $\Prob$-a.s. Thus we deduce $(\Delta U,U)=(\Delta V,V).$
\end{proof}
\begin{lem}\label{finite}
Let $(\Delta X,X)$ a $TM\times M$-valued process satisfying (i), (ii), (iii) in \dref{Pic3.1} and \eqref{eq:v}. Suppose that for all $\omega \in \Omega$ and $t>0$, $\Delta X$ satisfies
\[
\sum_{0<s\leq t}|\Delta X_s(\omega)|^2<\infty.
\]
Then for all $\gamma \in C_g$, $\omega \in \Omega$ and $t>0$, the number of $s\in [0,t]$ with
\[
\Delta X_s(\omega) \neq \gamma (X_{s-}(\omega),X_s(\omega))
\]
is finite. Furthermore, $(\Delta X,X)$ is a $\Delta$-semimartingale.
\end{lem}
\begin{proof}
Fix $t\geq 0$, $\omega \in \Omega$, and a connection rule $\gamma \in C_g$. Let $r:M\to [0,\infty]$ be an injective radius. Then $r$ is positive and continuous on $M$. Since $\overline{X(\omega, [0,t])}$ is compact, $r$ admits the minimum value $r_0$ on the set. Since $\displaystyle \sum_{0<s\leq t}|\Delta X_s(\omega)|^2<\infty$, the number of $s\in [0,t]$ with $|\Delta X_s(\omega)|\geq r_0$ is finite. For $s$ with $|\Delta X_s|<r_0$, we have $\Delta X_s(\omega)=\gamma (X_{s-}(\omega),X_s(\omega))$. Therefore the number of $s\in [0,t]$ with $\Delta X_s(\omega)\neq \gamma(X_{s-}(\omega),X_s(\omega))$ is finite. Furthermore, for all $T^*M$-valued c\`{a}dl\`{a}g processes $\phi$ above $X$, \eqref{iv} is satisfied and hence $(\Delta X,X)$ is a $\Delta$-semimartingale.
\end{proof}
\begin{proof}[Proofs of \tref{main} (1) and (3)]
By virtue of (2) of \tref{main} (2), any horizontal $\Delta$-semimartingale $(\Delta U,U)$ can be described as the development of an anti-development $W$ of $(\Delta U,U)$. Thus we deal with (1) and (3) in \tref{main} simultaneously. Since $\text{Exp}_{U_{s-}}\Delta W^k_s L_k=U_s$, $\Delta U$ is the initial velocity of the geodesic from $U_{s-}$ to $U_s$ with regard to $\tilde g$ and $\Delta X_s$ is the initial velocity of the geodesic from $X_{s-}$ to $X_s$. Since $[W,W]_t(\omega)<\infty$, $\displaystyle \sum_{0<s\leq t}|\Delta W_s(\omega)|^2<\infty$ for any fixed $t\geq 0$. Hence $\displaystyle \sum_{0<s\leq t} |\Delta X_s(\omega)|^2<\infty$, $\displaystyle \sum_{0<s\leq t} |\Delta U_s(\omega)|^2<\infty$ because
\[
|\Delta W|=|\Delta U|=|\Delta X|.
\]
Therefore by the previous lemma, $(\Delta X,X)$ and $(\Delta U,U)$ are $\Delta$-semimartingales. This completes the proof of \tref{main} (3). By \lref{intvec} and \pref{covariant} in \sref{bundlemetric},
\begin{align*}
\int \theta (U_-)\circ dU&=\int \langle \theta, L_k\rangle \circ dW^k=0,\\
\int \tilde \nabla \theta (U_-)\, d[U,U]^c&=\int \tilde \nabla \theta (L_k(U_-),L_l(U_-))\, d[W^k,W^l]^c=0.
\end{align*}
Then we obtain
\[
\int \theta (U_-)\, dU=0.
\]
Similarly, it holds that
\begin{align*}
\int \mathfrak{s}^k (U_-)\circ dU&=\int \langle \mathfrak{s}^k,L_l\rangle \circ dW^l=\int \delta^k_l \circ dW^l=W^k,\\
\int \tilde \nabla \mathfrak{s}(U_-)\, d[U,U]^c&=\int \tilde \nabla \mathfrak{s}(L_k(U_-),L_l(U_-))\, d[W^k,W^l]^c=0.
\end{align*}
Hence we obtain
\[
\int \mathfrak{s}(U_-)\, dU=W.
\]
This completes the proof of (1).
\end{proof}
\subsection{Proof of \tref{main} (4)}
In this subsection, we prove (4) dividing \tref{main} (4) into \tsref{4-1}, \ref{4-2}, and \ref{4-3}. Let $\gamma \in C_g$. (We can take such a connection rule by \pref{connection}.)
We set
\[
C:=\{ (x,u)\in M\times \mathcal{O}(M)\mid \pi u=x \}
\]
and define a map $\varphi:C\times M\to \mathcal{O}(M)$ by
\[
\varphi (x,u,y):=\tilde \eta_{x,y} (1),\ x,y\in M,\ u\in \mathcal{O}_x(M),
\]
where $\eta_{x,y} (t)$ is the geodesic with $\eta_{x,y,u} (0)=x,\ \eta_{x,y,u}'(0)=\gamma(x,y)$, and $\tilde \eta_{x,y,u}$ is a horizontal lift of $\eta$ whose initial value is $u$. Since $\gamma$ is measurable on $M\times M$ and differentiable on the diagonal set of $M\times M$, $\varphi$ is a constraint coefficient of SDE's. Suppose that an $M$-valued semimartingale $X$ is defined on $[0,\infty )$, that is, $X$ does not explode in finite time. Define the connection rule $\tilde \gamma$ on $(\mathcal{O}(M),\tilde g)$ as follows:
\begin{align*}
&\tilde \gamma(u,v)\\
&=\left\{ \begin{array}{ll}
\text{The initial velocity of minimal geodesic from}\ u\ \text{to}\ v,\ &(u,v)\in D_{\mathcal{O}(M)},\\
\left( \pi _*|_{H_u}\right)^{-1}\gamma(\pi u, \pi v),\ &(u,v)\in C_{\mathcal{O}(M)},
\end{array}\right.
\end{align*}
where
\begin{align*}
D_{\mathcal{O}(M)}&=\{ (u,v)\in \mathcal{O}(M)\times \mathcal{O}(M)\mid u\ \text{and}\ v\ \text{can be connected}\\
&\hspace{50mm} \text{by a unique geodesic with respect to}\ \tilde g\},\\
C_{\mathcal{O}(M)}&=\mathcal{O}(M)\times \mathcal{O}(M)\backslash D_{\mathcal{O}(M)},
\end{align*}
and $H_u$ is the horizontal subspace of $T_u\mathcal{O}(M)$ (See \sref{bundlemetric}). For later use, we start with three lemmas.
\begin{lem}\label{tildegamma}
For $u,v\in \mathcal{O}(M)$ and $a\in O(d)$,
\[
R_{a*}\tilde{\gamma}(u,v)=\tilde{\gamma}(ua,va),
\]
where
\[
R_a:\mathcal{O}(M)\to \mathcal{O}(M),\ R_au=ua,\ a\in O(d).
\]
\end{lem}
\begin{proof}
It holds that $(u,v)\in D_{\mathcal{O}(M)}\Leftrightarrow (ua,va)\in D_{\mathcal{O}(M)}$ for $u,v\in \mathcal{O}(M)$ and $a\in O(d)$ by \pref{geo} in \sref{bundlemetric}. First suppose $(u,v)\in D_{\mathcal{O}(M)}$. Set $\tau (t)=\exp_ut\tilde{\gamma}(u,v)$. Then $R_{a}\tau (t)$ is a unique minimal geodesic from $ua$ to $va$ by \pref{geo} in \sref{bundlemetric} again. Therefore
\[
R_{a*} \tilde \gamma(u,v)=R_{a*}\frac{d\tau}{dt}(0)=\frac{d}{dt}R_a\tau (0)=\tilde \gamma (ua,va).
\]
Next we suppose that $(u,v)\in C_{\mathcal{O}(M)}$, then $\tilde \gamma(u,v)$ is the horizontal lift of $\gamma (\pi u,\pi v)$ at $u$. Thus $R_{a*}\tilde \gamma(u,v)$ is the horizontal lift of $\gamma(\pi u,\pi v)$ at $ua$ and equals $\tilde \gamma (ua,va)$.
\end{proof}

\begin{lem}\label{horizon2}
Let $(\Delta U,U)$ be an $\mathcal{O}(M)$-valued $\Delta$-semimartingale satisfying \eqref{eq:v}. Then $(\Delta U,U)$ is horizontal if and only if it holds that
\begin{gather}
\int \theta(U_-)\, dU=0.\label{itohor}
\end{gather}
\end{lem}
\begin{proof}
By (1) and (3) of \tref{main}, if $(\Delta U,U)$ is horizontal, then \eqref{itohor} holds. Conversely, suppose that \eqref{itohor} holds. For any stopping times $S,T$ with $S\leq T$, it holds that
\begin{gather}
\tilde \nabla \theta (U_S)(\tilde \gamma(U_S,U_T),\tilde \gamma (U_S,U_T))=0 \label{zero1}
\end{gather}
by \pref{covariant} in \sref{bundlemetric}. Similarly, it holds that for $s\geq 0$,
\begin{equation}\label{zero2}
\begin{split}
\tilde \nabla \theta (U_{s-})(\tilde \gamma(U_{s-},U_s),\tilde \gamma(U_{s-},U_s))=0,\\
\tilde \nabla \theta (U_{s-})(\Delta U_s,\Delta U_s)=0.
\end{split}
\end{equation}
\eqref{zero1} and \eqref{zero2} imply
\[
\int \tilde \nabla \theta (U_-)\, d[U,U]=0.
\]
Therefore $(\Delta U,U)$ is horizontal.
\end{proof}
\begin{lem}\label{connectionhorizon}
Let $X$ be an $M$-valued semimartingale. Fix an $\mathcal{F}_0$-measurable $\mathcal{O}(M)$-valued random variable $u_0$ such that $u_0\in \mathcal{O}_{X_0}(M)$. Suppose that a semimartingale $U$ valued in $\mathcal{O}(M)$ is defined by the solution of the following SDE\\
\begin{equation}\label{connectionSDE}
\left\{ \begin{array}{ll}
\overset{\triangle}{d}U=\varphi(U,\overset{\triangle}{d}X),\\
U_0=u_0.
  \end{array} \right.
\end{equation}
Then $U$ does not explode in finite time with probability one and
\begin{gather}
\int \theta \circ \tilde \gamma dU=\int \theta \ \tilde \gamma dU=0\label{gamma1}.
\end{gather}
In particular, $(\tilde{\gamma} (U_-,U),U)$ is a horizontal lift of $(\gamma (X_-,X),X)$. Furthermore, it holds that
\begin{gather}
\int \frak{s} \circ \tilde \gamma dU=\int \frak{s} \ \tilde \gamma dU,\label{gamma2}\\
\int \frak{s}^k(U_-) \, \tilde \gamma dU=\int U_-\varepsilon ^k\ \gamma dX,\ k=1,\dots,d.\label{gamma3}
\end{gather}
\end{lem}
\begin{proof}
Since $U$ is the solution of \eqref{connectionSDE}, it holds that
\[
\varphi(X_{t-},U_{t_-},X_t)=U_t.
\]
This implies that $U_{t-}$ and $U_t$ can be connected by a horizontal minimal geodesic with respect to the metric $\tilde g$ and one of the minimal geodesics is the horizontal lift of $\exp_{X_{t-}} t\gamma (X_{t-},X_t)$ by \pref{geodesic}. Therefore $\tilde \gamma (U_{t-},U_t)$ is horizontal by the definition of $\tilde \gamma$. Let $\zeta$ be an explosion time of $U$ and assume
\[
\Prob (\zeta <\infty)>0.
\]
Then for $\omega \in \{ \zeta <\infty \}$, $\{ U_t(\omega) \}_{0\leq t<\zeta (\omega)}$ is not relatively compact. On the other hand, since $X$ does not explode in finite time, 
\[
A(\omega):=\{ X_t(\omega) \mid 0\leq t \leq \zeta(\omega) \}
\]
is relatively compact in $M$. Now it holds that $\{ U_t(\omega) \}_{0\leq t<\zeta (\omega)}\subset \pi ^{-1}(\overline{A(\omega)})$ and the right-hand side is compact because $O(d)$ is compact. This is a contradiction. Therefore $\zeta =\infty$, $\Prob$-a.s. Next we will show the second claim. Since $U$ is a solution of the SDE \eqref{connectionSDE}, for any $\overset{\triangle}{\mathbb{T}}\mathcal{O}(M)$-valued c\`{a}dl\`{a}g process $\Theta$, we have
\[
\int \Theta \ \overset{\triangle}{d}U=\int \Theta (\phi (X_-,U_-,\cdot)) \ \overset{\triangle}{d}X.
\]
In particular, we have
\begin{align*}
\int \theta (U_-)\ \tilde \gamma dU&=\int \theta_{U_-}(\tilde \gamma (U_-,\cdot))\ \overset{\triangle}{d}U\\
&=\int \theta_{U_-}(\tilde \gamma (U_-,\phi(X_-,U_-,\cdot)))\ \overset{\triangle}{d}X\\
&=0.
\end{align*}
Therefore by \lref{horizon2}, \eqref{gamma1} holds and consequently, $(\tilde \gamma (U_-,U),U)$ is a horizontal lift of $(\Delta X,X)$.
Therefore \eqref{gamma2} can be obtained by \tref{main} (1). Finally we will show \eqref{gamma3}.
We begin with the left-hand side of the claimed equation:
\begin{align*}
\int \mathfrak{s}^k(U_-)\ \tilde \gamma dU&=\int \pi^*(U_-\varepsilon^k) \ \tilde \gamma dU\\
&=\int \pi^*(U_-\varepsilon^k)(\tilde \gamma(U_-,\cdot))\ \overset{\triangle}{d}U\\
&=\int \pi^*(U_-\varepsilon^k)(\tilde \gamma (U_-,\phi (X_-,U_-,\cdot)))\ \overset{\triangle}{d}X\\
&=\int U_-\varepsilon^k(\gamma (X_-,\cdot ))\ \overset{\triangle}{d}X\\
&=\int (U_-\varepsilon^k)\ \gamma dX.
\end{align*}
Therefore we obtain \eqref{gamma3} and this completes the proof.
\end{proof}

Now \lref{connectionhorizon} guarantees the existence of horizontal lift of $\Delta$-semimartingales of the form $(\gamma(X_-,X),X)$ with $\gamma \in C_g$. Next we show the existence of the horizontal lift of $(\Delta X,X)$ with \eqref{eq:v}.
\begin{thm}\label{4-1}
Let $(\Delta X,X)$ be an $M$-valued $\Delta$-semimartingale and $u_0$ an $\mathcal{O}_{X_0}(M)$-valued $\mathcal{F}_0$-measurable random variable. Then there exists a horizontal lift of $(\Delta X,X)$ with $U_0=u_0$ satisfying \eqref{eq:v}.
\end{thm}
\begin{proof}
Fix a connection rule $\gamma \in C_g$. Let $U^{\gamma}$ be the horizontal lift of $(\gamma (X_-,X),X)$. If an $\mathcal{O}(M)$-valued semimartingale $U$ satisfies $\pi U=X$, there exists an $O(d)$-valued process $a_s$ such that $U=U^{\gamma}a$. We will specify the process $a_s$. For each $s\geq 0$, Set
\begin{align*}
c_s^1(t)&:=\exp t\gamma (X_{s-},X_s),\ t\in[0,1],\\
c_s^2(t)&:=\exp t\Delta X_s,\ t\in [0,1],\\
(c_s^1)^{-1}(t)&:=c_s^1(1-t),\ t\in [0,1],
\end{align*}
and
\begin{align*}
c_s^2\cdot \left( c_s^1\right)^{-1}=\left\{
\begin{array}{ll}
(c_s^1)^{-1}(2t),\ t\in [0,\frac{1}{2}],\\
c_s^2(2t-1),\ t\in [\frac{1}{2},1].
\end{array}\right.
\end{align*}
Denote by  $\widetilde{c_s^2\cdot \left( c_s^1\right)^{-1}}$ the horizontal lift of $c_s^2\cdot \left( c_s^1\right)^{-1}$ starting at $U_s^{\gamma}$. Then there exists a unique element $b_s\in O(d)$ satisfying
\begin{gather}
\widetilde{c_s^2\cdot \left( c_s^1\right)^{-1}}(1)=U_s^{\gamma}b_s.\label{bt}
\end{gather}
Since $b_s$ equals the unit element $e$ in $O(d)$ for $s\geq 0$ with $\gamma (X_{s-},X_s)=\Delta X_s$, $b_s$ equals the unit element except for finite number of $s\in [0,t]$ for fixed $t\geq 0$ by \lref{finite}. Let $0<T_1<T_2<\cdots$ be a sequence of stopping times which exhausts the time $s$ with $\gamma (X_{s-},X_s)\neq \Delta X_s$. We define $O(d)$-valued processes $\delta_s$, $a_s$ as follows:\\
\begin{equation}\label{at}
\begin{split}
a_s&=\delta _s=e,\ s\in [0,T_1),\\
\delta_s&=\left( a_{T_{i-1}}\right)^{-1}b_{T_i}a_{T_{i-1}},\ s\in [T_i,T_{i+1}),\\
a_s&=a_{T_{i-1}}\delta_{T_i},\ s\in [T_i,T_{i+1}).
\end{split}
\end{equation}
Then $a_t$ satisfies
\begin{gather}
a_t=\prod_{0<s\leq t} b_{t-s}\label{atbt}
\end{gather}
for each $t$. We set
\[
U_s:=U^{\gamma}_sa_s,
\]
and
\[
\Delta U_s:=(\text{the horizontal lift of }\Delta X_s\ \text{at }U_{s-}).
\]
Obviously we have $\pi U=X$ and $\pi_*\Delta U=\Delta X$. Let us prove that $(\Delta U,U)$ is a $\Delta$-semimartingale satisfying \eqref{eq:v}. Denote the horizontal lift of $\Delta X_s$ at $U_{s-}^{\gamma}$ as $\Delta U_s'$. Then
\[
\Delta U_s=R_{a_{s-}*}\Delta U_s'.
\]
By the definition of $b_s$, it holds that
\[
\exp_{U_{s-}^{\gamma}}\Delta U'_s=U_{s-}^{\gamma}b_s.
\]
Therefore we obtain
\begin{align*}
\exp_{U_{s-}}\Delta U_s&=\exp_{U_{s-}a_{s-}}\left( R_{a_{s-}*}\Delta U_s'\right) \\
&=\left( \exp_{U_{s-}^{\gamma}}\Delta U'_s\right) a_{s-}\\
&=U_s^{\gamma}b_s a_{s-}\\
&=U_s^{\gamma}a_s\\
&=U_s.
\end{align*}
At the second equality, we use \pref{geodesic} (2) in \sref{bundlemetric}. Thus $(\Delta U,U)$ satisfies \eqref{eq:v} and consequently it is a $\Delta$-semimartingale by \lref{finite}. Next we prove that $(\Delta U, U)$ is horizontal. It suffices to show that
\begin{align}
\int \theta(U_-)\, dU=0,\label{ti1}
\end{align}
by \lref{horizon2}. For each $i$, it is obvious that
\begin{gather}
\langle \theta (U_{T_i-}),\Delta U_{T_i}\rangle=0\label{ti2}
\end{gather}
by the definition of $\Delta U$. 
For $r,s\in (T_i,T_{i+1})$ with $r<s$, by \lref{tildegamma},
\begin{align*}
\langle \theta (U_{r}),\tilde \gamma (U_r,U_s)\rangle&=\langle \theta (U_r ^{\gamma}a_{T_i}),\tilde \gamma (U_r^{\gamma}a_{T_i}, U_s^{\gamma}a_{T_i})\rangle \\
&= \langle \theta (R_{a_{T_i}}U_r^{\gamma}),R_{a_{T_i}*}\tilde \gamma (U_r^{\gamma}, U_s^{\gamma})\rangle \\
&=Ad(a_{T_i}) \langle \theta  (U_r^{\gamma}),\tilde \gamma (U_r^{\gamma}, U_s^{\gamma})\rangle.
\end{align*}
Therefore it holds that
\begin{align}
\int _{(T_i,T_{i+1})} \theta (U_{s-})\, dU_s=Ad(a_{T_i})\int _{(T_i,T_{i+1})} \theta (U_s)\, \tilde \gamma dU_s=0 \label{ti3}
\end{align}
for each $i$. Combining \eqref{ti2} and \eqref{ti3}, we obtain \eqref{ti1}. Therefore we can deduce that $(\Delta U,U)$ is a horizontal lift of $(\Delta X,X)$.
\end{proof}
Next we show the uniqueness of the horizontal lift. Let $(\Delta U,U)$ be a horizontal lift of $(\Delta X,X)$ satisfying  $U_0=u_0$ and $\exp_{U_{s-}}\Delta U_s=U_s$. Let $\gamma \in C_g$ and denote the horizontal lift of $(\gamma(X_-,X),X)$ with an initial value $u_0$ by $U^{\gamma}$ again. Then there exists an $O(d)$-valued adapted process $a_t$ satisfying $U=U^{\gamma}a$ and such $a_t$ is unique since the action of $O(d)$ to each fiber of $\mathcal{O}(M)$ is free. Note that it holds that
\[
\Delta U_s=(\text{the horizontal lift of }\Delta X_s\ \text{at }U_{s-})
\]
by the definition of horizontal lifts. We will show that the process $a_t$ equals the process which is constructed in the proof of \tref{4-1}. We denote $U^{\gamma}$ by $V$ to simplify the notation.
\begin{lem}\label{lemat}
It holds that
\[
a_{s-}(\omega)\neq a_s(\omega)\Rightarrow \gamma (X_{s-}(\omega),X_s(\omega))\neq \Delta X_s (\omega)
\]
for $s\geq 0$, $\omega \in \Omega$.
\end{lem}
\begin{proof}
If $s$ and $\omega$ satisfy
\[
 \gamma (X_{s-}(\omega),X_s(\omega))=\Delta X_s (\omega),
\]
then each of $\Delta U_s(\omega)$ and $\Delta V_s(\omega)$ is the horizontal lift of $\gamma(X_{s-}(\omega),X_s(\omega))$ at $U_{s-}$ and $V_{s-}$, respectively. Therefore we have
\[
\exp t\Delta U_s(\omega)=\left( \exp t\Delta V_s(\omega)\right)a_{s-}(\omega),\ t\in [0,1].
\]
In particular, $U_s(\omega)=V_s(\omega)a_{s-}(\omega)$. On the other hand, the process $a_s$ satisfies $U_s(\omega)=V_s(\omega)a_s(\omega)$. Thus we can deduce $a_s(\omega)=a_{s-}(\omega)$.
\end{proof}
Applyng \lref{finite}, for any fixed $t\geq 0$ and $\omega \in \Omega$, the number of $s\in [0,t]$ with $a_s(\omega)\neq a_{s-}(\omega)$ is finite. Let $T_1<T_2<\cdots$ be a sequence of stopping times which exhausts the jumps of $a_t$. We can also observe that $\Delta U_s=\tilde \gamma (U_{s-},U_s)$, $s\in (T_i,T_{i+1})$. Next we will show that $a_s$ is constant on each $[T_i,T_{i+1})$. This can be shown in the same way as Theorem 3.2 in \cite{PE}.
\begin{lem}\label{lemj}
Suppose that $j$ is the canonical 1-form, which is a 1-form on $O(d)$ valued in $\mathfrak{o}(d)$ defined by
\[
j(fA)=A,\ f\in O(d),\ A\in \mathfrak{o}(d).
\]
Then it holds that
\begin{gather}
\int_{(T_i,T_{i+1})}j\circ da=0\label{canonicalform}
\end{gather}
and consequently, $a_t$ is constant on $(T_i,T_{i+1})$ for each $i$.
\end{lem}
\begin{proof}
Define $\Phi:\mathcal{O}(M)\times O(d)\to \mathcal{O}(M)$ and $\Phi_u:O(d)\to \mathcal{O}(M)$ by
\begin{align*}
\Phi(u,g)&=u\cdot g,\\
\Phi_u(g)&=u\cdot g.
\end{align*}
for $u\in \mathcal{O}(M)$. Then
\[
U_t=\Phi (V_t,a_t).
\]
Let $(u^{\alpha})$ be a local coordinate of $\mathcal{O}(M)$. Suppose that $V$ lives in the coordinate neighborhood on $[\sigma, \tau)\subset [T_i,T_{i+1})$, where $\sigma$ and $\tau$ are stopping times. We denote $V^{\alpha}=u^{\alpha}(V)$, and $U^{\alpha}=u^{\alpha}(U)$ on $[\sigma,\tau)$. Then by It\^o's formula,
\begin{align*}
U_t^{\alpha}-U_{\sigma}^{\alpha}=&\int_{(\sigma,t]}\left\{ \frac{\partial \Phi^{\alpha}}{\partial u^{\beta}}(V_{s-},a_s)\circ dV^{\beta}_s+\frac{\partial \Phi^{\alpha}}{\partial a^k}(V_{s-},a_s)\circ da^k_s \right\}\\
+&\sum_{\sigma<s\leq t}\left\{ \Phi^{\alpha}(V_s,a_s)-\Phi^{\alpha}(V_{s-},a_s)-\frac{\partial \Phi^{\alpha}}{\partial u^{\beta}}(V_{s-},a_s)\Delta V_s^{\beta} \right\}
\end{align*}
for $t\in (\sigma, \tau)$. Therefore by \pref{coordinate},
\begin{align}
\int_{(\sigma,t]}\theta \circ dU=&\int_{(\sigma,t]}\theta_{\alpha}\circ dU^{\alpha}+\sum_{\sigma<s\leq t}\left\langle \theta, \Delta U-\Delta U^{\alpha}\frac{\partial}{\partial u^{\alpha}}\right\rangle \nonumber \\
=&\int_{(\sigma,t]}\theta_{\alpha}\left\{ \frac{\partial \Phi^{\alpha}}{\partial u^{\beta}}(V_{s-},a_s)\circ dV^{\beta}_s+\frac{\partial \Phi^{\alpha}}{\partial a^k}(V_{s-},a_s)\circ da^k_s \right\} \nonumber \\
+&\sum_{\sigma<s\leq t} \theta_{\alpha}\left\{ \Phi^{\alpha}(V_s,a_s)-\Phi^{\alpha}(V_{s-},a_s)-\frac{\partial \Phi^{\alpha}}{\partial u^{\beta}}(V_{s-},a_s)\Delta V_s^{\beta} \right\} \nonumber \\ 
+&\sum_{\sigma<s\leq t}(\langle \theta,\Delta U\rangle -\theta_{\alpha}\Delta U^{\alpha}),\label{hor1}
\end{align}
where $(a^{\alpha})$ is a coordinate of $O(d)$ and $\theta=\theta_{\alpha}du^{\alpha}$ with $\theta_{\alpha}\in C^{\infty}(\mathcal{O}(M);\mathfrak{o}(d))$. Note that it holds that
\begin{align*}
\theta_{\alpha}\frac{\partial \Phi^{\alpha}}{\partial u^{\beta}}(u, a)&=\left(Ad(a^{-1})\theta \right)_{\beta}(u),\\
\theta_{\alpha}\frac{\partial \Phi^{\alpha}}{\partial a^k}(u,a)&= \left\langle d\left(\Phi_u^*\theta \right),\frac{\partial}{\partial a^k} \right\rangle = \left\langle j,\frac{\partial}{\partial a^k} \right\rangle.
\end{align*}
Therefore \eqref{hor1} can be rewritten as
\begin{align}
\int_{(\sigma,t]}\theta \circ dU=&\int_{(\sigma,t]}\left\{ \left( Ad(a^{-1})\theta \right)_{\beta}(V)\circ dV^{\beta}+\left\langle j,\frac{\partial}{\partial a^k} \right\rangle (a)\circ da^k \right\} \nonumber \\
&+\sum_{\sigma<s\leq t}\left( \langle \theta(U_{s-}), \Delta U_s\rangle -\left\langle Ad(a_s^{-1})\theta,\frac{\partial}{\partial u^{\beta}} \right\rangle \Delta V_s^{\beta} \right).\label{hor2}
\end{align}
Moreover, it holds that
\begin{align}
\int_{(\sigma, t]}Ad(a^{-1})(\theta)\circ dV=&\int_{(\sigma,t]}\left( Ad(a^{-1})\theta \right)_{\beta}\circ dV^{\beta} \nonumber \\
&+\sum_{\sigma<s\leq t}\{ \langle Ad(a_s^{-1})\theta,\Delta V_s\rangle-\langle Ad(a_s^{-1})\theta,\frac{\partial}{\partial u^{\beta}}\rangle \Delta V_s^{\beta} \} \label{hor3}
\end{align}
by \pref{coordinate} and
\begin{align}
\langle \theta (U_{s-}),\Delta U_s\rangle=&\langle \theta (V_{s-}a_s),R_{a_s*}\Delta V_s\rangle \nonumber \\
=&\langle R_{a_s}^*\theta (V_{s-}),\Delta V_s\rangle \nonumber \\
=& \langle Ad(a_s^{-1}) \theta (V_{s-}),\Delta V_s\rangle. \label{hor4}
\end{align}
Substituting \eqref{hor3} and \eqref{hor4} into \eqref{hor2}, we can deduce that
\begin{align*}
0=&\int_{(\sigma,t]}\theta \circ dU\\
=&\int_{(\sigma,t]}Ad(a^{-1})\theta \circ dV+\int_{(\sigma,t]}j\circ da\\
=&\int_{(\sigma,t]}j\circ da.
\end{align*}
Thus \eqref{canonicalform} follows and this implies that $a_s$ is constant on $(T_i,T_{i+1})$ for each $i$.
\end{proof}
By \lref{lemj}, we obtain the following uniqueness result.
\begin{thm}\label{4-2}
$(\Delta U,U)$ is uniquely determined.
\end{thm}
\begin{proof}
It suffices to show that $a_t$ equals the process defined in \eqref{at}. Let $b_t$ be an $O(d)$-valued process defined through \eqref{bt}. Denote the horizontal lift of $\Delta X_s$ at $V_s$ by $\Delta U'_s$. Then for $s\geq 0$,
\begin{align*}
V_s a_s&=U_s\\
&=\exp_{V_{s-}a_{s-}} \left(R_{a_{s-}*}\Delta U_s'\right) \\
&=\left( \exp_{V_{s-}}\Delta U'_s\right) a_{s-}\\
&=V_sb_sa_{s-}.
\end{align*}
Since $a_t$ is constant on $(T_i,T_{i+1})$ for each $i$, it holds that
\[
a_t=\prod_{0\leq s<t}b_{t-s},\ t\geq 0.
\]
This completes the proof by \eqref{atbt}.
\end{proof}

We end the proof of \tref{main} (4) with the following theorem.
\begin{thm}\label{4-3}
It holds that
\[
W_t^i=\int_0^tU_{s-}\varepsilon ^i\, dX_s,
\]
where $W$ is the anti-development of $(\Delta U,U)$.
\end{thm}
\begin{proof}
Since 
\[
\int_{(0,t]}U_{s-}\varepsilon^i\ dX=\sum_{m=1}^{\infty}\int_{(T_m\wedge t,T_{m+1}\wedge t]}U_{s-}\varepsilon^i\ dX,
\]
it suffices to show
\begin{align}
\int_{(T_m\wedge t,T_{m+1}\wedge t)}U_{s-}\varepsilon^i\ dX&=\int _{(T_m\wedge t,T_{m+1}\wedge t)}\mathfrak{s}^i(U_{s-})\ dU_s, \label{anti1}\\
\langle U_{T_m\wedge t-}\varepsilon^i, \Delta X_{T_m\wedge t}\rangle&= \langle \mathfrak{s}^i(U_{T_m\wedge t-}),\Delta U_{T_m\wedge t}\rangle \label{anti2}
\end{align}
for each $m$. \eqref{anti2} can be easily obtained by the definition of the solder form. Thus we will show \eqref{anti1}. Set $c_j^i(s)=a_i^j(s)$.
Since $a_t$ is constant and $\Delta X=\gamma (X_-,X)$ on $(T_m,T_{m+1})$ for each $m$, we have
\begin{align}
\int_{(T_m\wedge t,T_{m+1}\wedge t)}U_{s-}\varepsilon^i\ dX_s&=\int_{(T_m\wedge t,T_{m+1}\wedge t)}U_{s-}\varepsilon^i \ \gamma dX_s \nonumber\\
&=\int_{(T_m\wedge t,T_{m+1}\wedge t)}\left(U_{s-}^{\gamma}a\right) \varepsilon^i \ \gamma dX_s\nonumber\\
&=c^i_j\int_{(T_m\wedge t,T_{m+1}\wedge t)}U_{s-}^{\gamma} \varepsilon^j\ \gamma dX_s\nonumber\\
&=c^i_j\int_{(T_m\wedge t,T_{m+1}\wedge t)}\mathfrak{s}^j(U_{s-}^{\gamma})\ \tilde{\gamma}dU^{\gamma}_s,\label{anti3}
\end{align}
by \eqref{gamma3}, where we write $c_j^i=c_j^i(T_m\wedge t)$, $a=a(T_m\wedge t)$ to simplify the notation. 
Furthermore, for stopping times $S,T$ with $T_i<S\leq T<T_{i+1}$, it holds that
\begin{align*}
c^i_j\langle \mathfrak{s}^j(U_S^{\gamma}), \tilde \gamma (U_S^{\gamma},U_T^{\gamma})\rangle=& c_j^i\langle \mathfrak{s}^j(U_{S}^{\gamma}),R_{a^{-1}*}R_{a*}\tilde \gamma (U_{S}^{\gamma},U_{T}^{\gamma})\rangle \\
=&\langle \mathfrak{s}^i (U_{S}),\tilde \gamma(U_{S},U_{T})\rangle
\end{align*}
by \lref{tildegamma}. Therefore it holds that
\begin{align}
c^i_j \int_{(T_m\wedge t,T_{m+1}\wedge t)}\mathfrak{s}^j(U_{s-}^{\gamma})\ \tilde \gamma dU_s^{\gamma}= \int _{(T_m\wedge t,T_{m+1}\wedge t)}\mathfrak{s}^i(U_{s-})\ dU_s.\label{anti4}
\end{align}
Combining \eqref{anti3} and \eqref{anti4}, we obtain \eqref{anti1} and the assertion follows.
\end{proof}
Combining \tsref{4-1}, \ref{4-2} and \ref{4-3}, we complete the proof of \tref{main} (4).
\subsection{Proof of \tref{main2}}
We end this section with the proof of \tref{main2} and its example. The idea of the proof of \tref{main2} is to construct a coefficient of SDE from a given connection rule $\gamma$, which was also considered in section 10.2 of \cite{Maillard03}. In the proof of \tref{main2} below, we give a concrete construction of the coefficient of SDE through the orthonormal frame bundle but we first need to construct the constants $\delta_0$ and $\delta >0$ since we do not impose any conditions on the connection rule.\\

First we determine the constants $\delta_0$, $\delta >0$ and the function $h \in C^{\infty}(\mathcal{O}(M)\times B_{\delta}(0); \mathbb{R}^d)$ appearing in \tref{main2}. Let $\gamma$ be an arbitrary connection rule which induces Levi-Civita connection and fix $\gamma^g\in C_g$. By the assumption for $\gamma$ and the compactness of $M$, there exists a neighborhood $\mathcal{U}$ of the diagonal set $\mathrm{diag}(M)$ of $M$ such that
\begin{align}\label{gammaexp}
\gamma(x,y)-\gamma^g(x,y)=O(d(x,y)^3)
\end{align}
on $\mathcal{U}$. Moreover, it is well known that the derivative of the map $TM \ni u\mapsto (\pi_{TM}u,\mathrm{exp}u)\in M \times M$ at the point $0\in T_xM$ equals
\begin{align*}
\left[
\begin{array}{ll}
\mathrm{id}_{T_xM} &\mathrm{id}_{T_xM}\\
\mathrm{id}_{T_xM} &0
\end{array}
\right].
\end{align*}
Thus by \eqref{gammaexp}, for any $x\in M$,
\begin{align*}
\gamma_{*(x,x)}=
\left[
\begin{array}{ll}
\mathrm{id}_{T_xM} &\mathrm{id}_{T_xM}\\
0 &-\mathrm{id}_{T_xM}
\end{array}
\right].
\end{align*}
Thus there exists $r>0$ such that $\gamma$ is a diffeomorphism on $B^M_r(x) \times B^M_r(x)$ into its image, where $B^M_r(x)$ is the geodesic ball on $M$ centered at $x$ with radius $r$. Let $R_x$ be the supremum of such $r>0$. We will show that the function $x \mapsto R_x$ is lower semi-continuous. Let $\{x_n\}_{n\in \mathbb{N}}$ be a sequence converging to $x$. Take $0<\ep < R_x$. Then for any sufficiently large $n$, $B^M_{R_x-\ep}(x_n) \subset B^M_{R_x}(x)$. Thus $\gamma$ is a diffeomorphism on $B^M_{R_x-\ep}(x_n) \times B^M_{R_x-\ep}(x_n)$. This means that $R_x - \ep \leq R_{x_n}$. Since $\ep>0$ is arbitrary, we have $\dis R_x \leq \liminf_{n\to \infty} R_{x_n}$ and this means that $x\mapsto R_x$ is lower semi-continuous. In particular, $R_x$ attains its minimum $R_0>0$ on $M$. Let $r_0$ be the injective radius of $M$. We set $\delta_0:= R_0 \land r_0$. We further set
\[
\delta_x:=\sup \{ \delta>0 \mid B^{T_xM}_{\delta}(0) \subset \gamma_x(B^M_{\delta_0}(x)) \},
\]
where $B^{T_xM}_{\delta}(0)$ is the ball on $T_xM$ centered at the origin with radius $\delta$. Then the map $x \mapsto \delta_x$ can also be shown to be lower semi-continuous as follows. Assume that there exist a point  $x\in M$ and a sequence $\{x_n \}$ such that
\[
\lim_{n\to \infty} x_n = x,\ \liminf_{n\to \infty} \delta_{x_n}< \delta_x.
\]
By taking a proper subsequence, we suppose that  $\dis \lim_{n\to \infty}\delta_{x_n}$ exists and satisfies $\dis \lim_{n\to \infty}\delta_{x_n}<\delta_x$. Let $\dis \ep \in (0,\delta_x-\lim_{n\to \infty}\delta_{x_n})$. Then there exists $N_{\ep}\in \mathbb{N}$ such that $\delta_{x_n}\leq \delta_{x}-\ep$ for any $n\geq N_{\ep}$. Then for each $n$, we can take $v_n \in B_{\delta_{x_n}}^{T_{x_n}M}(0) \cap \gamma_{x_n}(B^M_{\delta_0}(x_n))^c$. We set
\[
U^qM:=\{v\in TM \mid |v|\leq q\}
\]
for $q>0$. Since $U^{\delta_x-\ep}M$ is compact, there exists a subsequence $\{n_k\}_{k\in \mathbb{N}}$ such that $v_{n_k}$ converges to $v \in U^{\delta_x-\ep}_xM$ with respect to a proper metric which is compatible to the topology on $TM$. Since $\gamma_x$ is a diffeomorphism on $B^M_{\delta_0}(x)$, we can take $\ep'>0$ such that $B^{T_xM}_{\delta_x-\ep}(0) \subset \gamma_x(B^M_{\delta_0-\ep'}(x))$. This implies that $\gamma (B^M_{\delta_0-\ep'}(x), B^M_{\delta_0-\ep'}(x))$ is an open neighborhood of $v$ in $TM$. Therefore, there exists $K_1 \in \mathbb{N}$ such that $v_{n_k} \in \gamma (B^M_{\delta_0-\ep'}(x), B^M_{\delta_0-\ep'}(x))$ for all $k \geq K_1$. On the other hand, there exists $K_2\in \mathbb{N}$ such that $B^M_{\delta_0-\ep'}(x) \subset B^M_{\delta_0}(x_{n_k})$ for all $k \geq K_2$. Thus for $k \geq K_1 \lor K_2$, $v_{n_k} \in \gamma_{x_{n_k}}(B^M_{\delta_0}(x_{n_k}))$. This contradicts to the choice of $v_n$. Therefore, we obtain the lower semi-continuity of $\delta_x$ and consequently, $\delta_x$ also attains its minimum on $M$. Let $\dis \delta:=\min_{x\in M}\delta_x>0$. Then $\left(\gamma_x|_{B_{\delta_0}(x)}\right)^{-1}\left(B^{T_xM}_{\delta}(0)\right) \subset B^M_{\delta_0}(x)$ and $\gamma_x$ is a diffeomorphism on $\left(\gamma_x|_{B_{\delta_0}(x)}\right)^{-1}\left(B^{T_xM}_{\delta}(0)\right)$. Moreover, $\exp_x$ is a diffeomorphism on $\exp_x^{-1}\left(\left(\gamma_x|_{B_{\delta_0}(x)}\right)^{-1}\left(B^{T_xM}_{\delta}(0)\right)\right)$ since this set is included in $B^{T_xM}_{\delta_0}(0)$ and $\delta_0 \leq r_0$. Thus we can define
\[
b_x:= \gamma (x, \mathrm{exp}_x (\cdot)) \colon \exp_x^{-1}\left(\left(\gamma_x|_{B_{\delta_0}(x)}\right)^{-1}\left(B^{T_xM}_{\delta}(0)\right)\right) \to B^{T_xM}_{\delta}(0)
\]
and $b_x$ is a diffeomorphism. Define
\[
a(u,z):= b_{\pi u}^{-1}(uz)
\]
for $u\in \mathcal{O}(M)$ and $z\in B_{\delta}(0)$. Then $a$ satisfies
\[
\gamma (\pi u, \mathrm{exp}_{\pi u} a(u,z))=uz
\]
for all $u\in \mathcal{O}(M)$ and $z\in B_{\delta}(0)$ and $a(u,0)=0$. Moreover, if we let $h(u,z)=u^{-1}a(u,z)\in B_{\delta_0}(0)$ for $(u,z)\in \mathcal{O}(M)\times B_{\delta}(0)$, then $h$ is differentiable on $\mathcal{O}(M) \times B_{\delta}(0)$ and satisfies
\[
d_0h(u,\cdot)=\mathrm{id}_{\mathbb{R}^d}.
\]
Next, we show that
\begin{align}
\frac{\partial^2 h(u,\cdot)}{\partial z^i \partial z^j}(0)=0 \label{hess0}
\end{align}
for any $u\in \mathcal{O}(M)$. Fix $x\in M$ and we write $G_1(y):=\gamma (x,y)$ and $G_2(y):=\gamma^g(x,y)$. Then $a=G_2\circ G^{-1}_1$ on $B^{T_xM}_{\delta}(0)$. Let $(y^1,\dots,y^d)$ be a normal coordinate with $y^i(x)=0$ for $i=1,\dots,d$ and $(z^1,\dots,z^d)$ a coordinate on $T_xM$ for an orthonormal basis. Then we regard $G_1$ and $G_2$ as functions of $(y^1,\dots,y^d)$. By Taylor's theorem, we have
\[
G_1(y)-G_2(y)=\frac{1}{2}\left(\frac{\partial^2 G_1}{\partial y^i \partial y^j}(0)-\frac{\partial^2 G_1}{\partial y^i \partial y^j}(0)\right) y^iy^j+\mathrm{o}(|y|^3)
\]
for any $y$ near $x$. Since both $\gamma$ and $\gamma^g$ induce Levi-Civita connection, this implies $\mathrm{Hess}G^i_1(0)=\mathrm{Hess}G^i_2(0)$ for each $i=1,\dots,d$. Therefore, we have
\begin{align*}
\frac{\partial^2 G_2\circ G_1^{-1}}{\partial z^i \partial z^j}(0)=0.
\end{align*}
This immediately yields \eqref{hess0}. We set $\varphi \colon \mathbb{R}^d \times \mathcal{O}(M) \times \mathbb{R}^d \to \mathcal{O}(M)$ by \eqref{coefficient}. Then obviously, the map $\varphi$ is a constraint coefficient from $\mathbb{R}^d \times \mathcal{O}(M) \times \mathbb{R}^d$ to $\mathcal{O}(M)$.
\begin{proof}[Proof of (1) of \tref{main2}]
For a given $\mathcal{F}_0$-measurable random variable $U_0$ and a semimartingale $Z$, we obtain the unique solution $U$ of the SDE \eqref{modifiedSDE}.
We set
\[
\Delta U_t:=h^{\alpha}(U_{t-},\Delta Z_t) L_{\alpha}(U_{t-}).
\]
Then $(\Delta U, U)$ is a horizontal $\Delta$-semimartingale on $\mathcal{O}(M)$. In fact, by \lref{intvec} and \eqref{hess0}, we have
\begin{align*}
\int \theta (U)\circ dU&=\int \langle \theta, L_{\alpha} \rangle (U_-)\circ dZ^{\alpha}
+\sum_{0<s\leq \cdot}R^{\alpha}_s\langle \theta, L_{\alpha} \rangle(U_{s-}) \\
&=0,
\end{align*}
where
\[
R^{\alpha}_t:=h^{\alpha}(U_{t-},\Delta Z_t)-\Delta Z^{\alpha}_t.
\]
Let $X=\pi U$ and $\Delta X=\pi_*\Delta U$. Then $X$ satisfies
\begin{align*}
\gamma(X_{t-},X_t)&=\gamma(\pi U_{t-},\mathrm{exp}_{\pi U_{t-}}h^{\alpha}(U_{t-},\Delta Z_t)U_{t-}\ep_{\alpha})\\
&=\gamma(\pi U_{t-},\mathrm{exp}_{\pi U_{t-}}a(U_{t-},\Delta Z_t))\\
&=U_{t-}\Delta Z_t.
\end{align*}
We set
\[
W_t:=Z_t+\sum_{0<s\leq t}U_{s-}^{-1}(\gamma^g(X_{s-},X_s)-\gamma(X_{s-},X_s)).
\]
Then $\Delta W_t=U_t^{-1}\gamma^g(X_{t-},X_t)$ and it satisfies
\begin{align*}
\exp_{X_{t-}}U_{t-}\Delta W_t=X_t=\exp_{X_{t-}}a(U_{t-},\Delta Z_t).
\end{align*}
Thus we have
\[
U_{t-}\Delta W_t=a(U_{t-},\Delta Z_t).
\]
Therefore, for all $F\in C^{\infty}(\mathcal{O}(M))$, it holds that
\begin{align*}
F(U_t)-F(U_0)&=\int_0^tL_{\alpha}F(U_{s-})\circ dZ^{\alpha}_s\\
&\h + \sum_{0<s\leq t}\{F(\mathrm{Exp}_{U_{s-}}(h^{\alpha}(U_{s-},\Delta Z_s)L_{\alpha}))-F(U_{s-})-L_{\alpha}F(U_{s-})\Delta Z^{\alpha} \}\\
&=\int_0^tL_{\alpha}F(U_{s-})\circ dW^{\alpha}_s\\
&\h + \sum_{0<s\leq t}\{F(\mathrm{Exp}_{U_{s-}}(\Delta W^{\alpha}_tL_{\alpha}))-F(U_{s-})-L_{\alpha}F(U_{s-})\Delta W^{\alpha} \}.
\end{align*}
This implies that $W$ is an anti-development of $X$ with respect to the horizontal lift $(\Delta U,U)$. Therefore, we have
\begin{align*}
\int \phi_-\, \gamma dX &=\int \phi_-\, \gamma^g dX+\sum_{0<s\leq \cdot}\langle \phi_{s-}, \gamma(X_{s-},X_s)-\gamma^g(X_{s-},X_s) \rangle \\
&=\int \langle U_{-}^{-1}\phi_-, dW\rangle+\sum_{0<s\leq \cdot}\langle U_{s-}^{-1}\phi_{s-}, \Delta Z_s-\Delta W_s \rangle \\
&=\int \langle U_{-}^{-1}\phi_-, dZ\rangle
\end{align*}
for any $T^*M$-valued \cl process $\phi$ above $X$.
\end{proof}

\begin{proof}[Proof of (2) of \tref{main2}]
Let $X$ be an $M$-valued semimartingale satisfying
\[
X_t\in \left(\gamma_{X_{t-}}|_{B^M_{\delta_0}(X_{t-})}\right)^{-1}\left(B^{T_{X_{t-}}M}_{\delta}(0)\right)\ \text{for all}\ t\geq 0,\ \Prob \text{-a.s.}
\]
Let $W$ be an anti-development of $(\gamma^g(X_-,X),X)$ and $V$ a horizontal lift of $(\gamma^g(X_-,X),X)$ with an initial value $U_0$. We set
\[
Z_t:=W_t+\sum_{0<s\leq \cdot}V_{s-}^{-1}(\gamma(X_{s-},X_s)-\gamma^g(X_{s-},X_s)).
\]
Then by the assumption for $X$,
\begin{align*}
|\Delta W_t|&=|\gamma^g(X_{t-},X_t)<\delta_0,\\
|\Delta Z_t|&=|\gamma(X_{t-},X_t)|<\delta.
\end{align*}
Thus we have
\begin{align*}
\gamma(X_{s-},\exp_{X_{s-}}a(V_{s-},\Delta Z_s))&=V_{s-}\Delta Z_s\\
&=\gamma (X_{s-},X_s)\\
&=\gamma (X_{s-},\exp_{X_{s-}}V_{s-}\Delta W_s)
\end{align*}
and consequently,
\[
a(V_{s-},\Delta Z_s)=V_{s-}\Delta W_s.
\]
Therefore, $V$ satisfies
\begin{align*}
F(V_t)-F(V_0)&=\int_0^t L_{\alpha}F(V_{s-})\circ dW^{\alpha}_s\\
&\h +\sum_{0<s\leq t}\{F(\mathrm{Exp}_{V_{s-}}(\Delta W^{\alpha}_sL_{\alpha}))-F(V_{s-})-L_{\alpha}F(V_{s-})\Delta W^{\alpha}_s\}\\
&=\int_0^t L_{\alpha}F(V_{s-})\circ dZ^{\alpha}_s\\
&\h + \sum_{0<s\leq t}\{F(\mathrm{Exp}_{V_{s-}}(h^{\alpha}(V_{s-},\Delta Z_s)L_{\alpha}))-F(V_{s-})-L_{\alpha}F(V_{s-})\Delta Z^{\alpha}_s \}
\end{align*}
for all $F\in C^{\infty}(\mathcal{O}(M))$. This implies that $V$ solves \eqref{modifiedSDE} and we have $V=U$ by the uniqueness of the solution of the SDE. Therefore, $(U_-h(U_-,\Delta Z),U)$ is the horizontal lift of $X$ with the initial value $U_0$. In particular, the semimartingale $X$ satisfies \eqref{intconne} by the claim of (1).
\end{proof}

\begin{ex}\label{spmar}
We consider the case
\[
M=\mathbb{S}^d:=\{ (x^1,\dots,x^{d+1})\in \mathbb{R}^{d+1} \mid (x^1)^2+\cdots +(x^{d+1})^2=1\}.
\]
Let $\gamma$ be a connection rule on $\mathbb{S}^d$ given by
\[
\gamma (x,y):=\Pi_x(y-x),
\]
where $\Pi_x \colon \mathbb{R}^{d+1}\to T_x\mathbb{S}^d$ is the orthonormal projection. Then it holds that
\[
\gamma(x,y)=\frac{\sin d(x,y)}{d(x,y)}\gamma^g(x,y).
\]
Therefore, we can easily check that $\delta_0=\frac{\pi}{2}$, $\delta =1$ and
\[
a(u,z)= \frac{\arcsin |z|}{|z|} uz.
\]
We define $g \colon [0,1]\to \mathbb{R}$ by
\begin{align*}
g(\theta)=\left\{
\begin{array}{ll} \frac{\arcsin \theta-\theta}{\theta},\ &\theta \neq 0,\\
0,\ &\theta = 0.
\end{array}
\right.
\end{align*}
Then for a local martingale $Z$ on $\mathbb{R}^d$ with $Z_0$ and $\dis \sup_{0\leq t<\infty}|\Delta Z_t|<1$, if we set
\[
W=Z+\sum_{0<s\leq \cdot}g(|\Delta Z_s|)\Delta Z_s,
\]
then the development of $W$ is an $\mathbb{S}^d$-valued $\gamma$-martingale.
\end{ex}
\section{Appendix: The Riemannian metric and the Levi-Civita connection on $\mathcal{O}(M)$}\label{bundlemetric}
In order to define the stochastic integrals on $\mathcal{O}(M)$, we introduce the Riemannian metric and the Levi-Civita connection on $\mathcal{O}(M)$. First we recall the notion of orthonormal frame bundles. Fundamental properties of orthonormal frame bundles mentioned in this section are based on \cite{KN}. Let $(M,g)$ be a Riemannian manifold. Set
\begin{align*}
\mathcal{O}_x(M)&:=\{ u:\mathbb{R}^d\to T_xM\mid u\text{ is a linear isometric}\},\\
\mathcal{O}(M)&:=\bigsqcup_{x\in M}\mathcal{O}_x(M),\\
\pi&: \mathcal{O}(M)\to M,\ \pi(u)=x,\ u\in \mathcal{O}_x(M).
\end{align*}
Let $O(d)$ be an orthogonal group and $\mathfrak{o}(d)$ its Lie algebra. $O(d)$ acts on $\mathcal{O}(M)$ and the action is defined by
\[
ua:=u\circ a\in \mathcal{O}_x(M),\ x\in M,\ u\in \mathcal{O}_x(M),\ a\in O(d).
\]
We define a map $R_a:\mathcal{O}(M)\to \mathcal{O}(M)$ by
\[
R_au:=ua,\ u\in \mathcal{O}(M)
\]
for $a\in O(d)$. It is well known that $\mathcal{O}(M)$ is a differentiable manifold and  $\pi :\mathcal{O}(M)\to M$ is a principal $O(d)$-bundle. For each $u\in \mathcal{O}(M)$, $\ker \pi_{*u}$ is called the vertical subspace of $T_u\mathcal{O}(M)$ and each vector in $\ker \pi_{*u}$ is said to be vertical. For $X\in \mathfrak{o}(d)$, the matrix exponential $\exp tX$ determines  a one parameter subgroup of $O(d)$. The vertical vector field $X^{\sharp}$ is defined by
\begin{align}\label{verticalvec}
X^{\sharp}(u)=\left( \frac{d}{dt}\right)_{t=0}u\exp tX,\ u\in \mathcal{O}(M).
\end{align}
Fix an orthonormal basis $\{X_{\alpha} \}_{\alpha =1,\dots, \frac{d(d-1)}{2}}$ of $\mathfrak{o}(d)$ with respect to the standard inner product. Then $\{X_{\alpha}^{\sharp}(u) \}_{\alpha =1,\dots, \frac{d(d-1)}{2}}$ is the basis of $\ker \pi_{*u}$ for each $u\in \mathcal{O}(M)$.\\
\indent By pulling back to $\mathcal{O}(M)$, a $(p,q)$-tensor $T$ can be seen as an $(\mathbb{R}^d)^p\otimes ((\mathbb{R}^{d})^*)^q$-valued function. The pullback of $T$ is defined by
\[
\pi^*T(u)=u^{-1}T_{\pi u}.
\]
This is called the scalarization of $T$. The tensor transformation rule can be described as
\[
\pi^*T(ua)=a^{-1}\pi^*T(u),\ u\in \mathcal{O}(M),\ a\in O(d).
\]
Let $\theta$ be a connection form on $\mathcal{O}(M)$, that is, $\theta$ is an $\mathfrak{o}(d)$-valued 1-form satisfying the following:
\begin{itemize}
\item[(1)] for all $X\in \mathfrak{o}(d)$, $\langle \theta,X^{\sharp}\rangle=X$,
\item[(2)] for all $g\in O(d)$, $R_g^*\theta=Ad(g^{-1})\circ \theta$,
\end{itemize}
where $R_g^*\theta$ is the pull-back of $\theta$ by $R_g$ and $Ad:O(d)\to \text{End}(\mathfrak{o}(d))$ is the adjoint representation. For a connection form $\theta$, set
\[
H_u=\{ A\in T_u\mathcal{O}(M)\mid \langle \theta,A\rangle=0\},\ u\in \mathcal{O}(M).
\]
This is called the horizontal subspace of $T_u\mathcal{O}(M)$. The restriction of $\pi_{*u}$ to $H_u$ denoted by $\pi_{*\mid H_u}:H_u\to T_{\pi u}M$ is a linear isomorphism for each $u\in \mathcal{O}(M)$ and for $X\in T_{\pi u}M$, the vector
\[
\tilde X=(\pi_{*\mid H_u})^{-1}(X)
\]
is called the horizontal lift of $X$. It is well known that the horizontal lift of a $C^{\infty}$-vector field on $M$ is a $C^{\infty}$-vector field on $\mathcal{O}(M)$. A connection form $\theta$ induces a connection on $M$. In fact,
\[
\nabla_X Y(\pi u)=u(\tilde X\pi^*Y(u)),\ u\in \mathcal{O}(M),
\]
determines a connection on $M$. We can observe that a covariant derivative of a vector field on $M$ can be seen as a derivative of a vector-valued function on $\mathcal{O}(M)$ through scalarization.
\begin{dfn}
The solder form $\mathfrak{s}\in \Omega^1 (\mathcal{O}(M);\mathbb{R}^d)$ is defined by
\[
\mathfrak{s}_u(A)=u^{-1}\pi_*A,\ u\in \mathcal{O}(M),\ A\in T_u\mathcal{O}(M).
\]
\end{dfn}
\begin{dfn}
The horizontal vector fields on $\mathcal{O}(M)$ defined by
\[
L_k(u)=(\pi_{*\mid H_u})^{-1}(u\varepsilon_k)\ (k=1,\dots,d)
\]
are called canonical horizontal vector fields.
\end{dfn}
Fix orthonormal bases of $\mathbb{R}^d$ and $\mathfrak{o}(d)$ and denote them by $\{ \varepsilon_i\}$ and $\{ X_{\alpha} \}$, respectively. Then $\theta$ and $\mathfrak{s}$ can be written as
\begin{gather*}
\theta=\theta^{\alpha}X_{\alpha},\ \mathfrak{s}=\mathfrak{s}^k\varepsilon_k.
\end{gather*}
The set of vector fields $\{ L_k, X_{\alpha}^{\sharp} \}^{k=1,\dots,d}_{\alpha =1,\dots, \frac{d(d-1)}{2}}$ is a basis of each tangent space on $\mathcal{O}(M)$ and $\{ \mathfrak{s}^k,\theta^{\alpha}\}$ is its dual basis.
\begin{dfn}\label{defmetric}
Let $\nabla$ be a connection on $M$, and $\theta$ a connection form which corresponds to $\nabla$. Let $\mathfrak{s}$ be a solder form. Define the Riemannian metric $\tilde g$ on $\mathcal{O}(M)$ by
\[
\tilde g=\sum_{\alpha}\theta^{\alpha}\otimes \theta^{\alpha}+\sum_{k}\mathfrak{s}^k\otimes \mathfrak{s}^k.
\]
Denote the Levi-Civita connection corresponding to $\tilde g$ by $\tilde \nabla$.
\end{dfn}
The Riemannian metric $\tilde g$, the Levi-Civita connection $\tilde \nabla$ and geodesics on $\mathcal{O}(M)$ are considered in \cite{PE}. Covariant derivatives with respect to connections on soldered principal fiber bundles are computed in \cite{BB} under a more general situation. Set
\begin{gather*}
\tilde \nabla L_j=\sum_{\beta}\omega_j^{\beta}\ X^{\sharp}_{\beta}+\sum_{i}\omega_j^i\ L_i,\\
\tilde \nabla X^{\sharp}_{\alpha}=\sum_{\beta} \omega_{\alpha}^{\beta}\ X^{\sharp}_{\beta}+\sum_i\omega_{\alpha}^i\ L_i.
\end{gather*}
Since $\tilde \nabla$ is torsion-free, it holds that
\[
\omega_j^{\alpha}=-\omega_{\alpha}^j.
\]
We can write
\begin{gather*}
\omega_{\beta}^{\alpha}=\sum_{\gamma}\omega_{\beta \gamma}^{\alpha}\ \theta^{\gamma}+\sum_k\omega_{\beta k}^{\alpha}\ \mathfrak{s}^k,\\
\omega_j^{\alpha}=\sum_{\gamma}\omega_{j \gamma}^{\alpha}\ \theta^{\gamma}+\sum_k\omega_{j k}^{\alpha}\ \mathfrak{s}^k,\\
\omega_j^i=\sum_{\gamma}\omega_{j \gamma}^i\ \theta^{\gamma}+\sum_k\omega_{jk}^i\ \mathfrak{s}^k.
\end{gather*}
In view of calculations in \cite[p.\ 897]{BB}, it holds that
\begin{align}
\omega_{\beta}^{\alpha}&=\frac{1}{2}\sum_{\gamma}c_{\beta \gamma}^{\alpha}\ \theta^{\gamma},\nonumber \\
\omega_j^{\alpha}&=-\frac{1}{2}\sum_k\Omega_{j k}^{\alpha}\ \mathfrak{s}^k,\nonumber \\
\omega_j^i&=\sum_{\gamma}\{ (X_{\gamma}^{\sharp})^{ij}-\frac{1}{2}\Omega_{ij}^{\gamma} \} \theta^{\gamma}\label{co},
\end{align}
where
\[
c_{\alpha \beta}^{\gamma},\ \alpha,\beta,\gamma=1,\dots, \frac{d(d-1)}{2}
\]
are the structure constants with respect to an orthonormal basis $\{X_{\alpha} \}_{\alpha =1,\dots, \frac{d(d-1)}{2}}$ of $\mathfrak{o}(d)$ defined through
\[
[X_{\alpha},X_{\beta}]_{\mathfrak{o}(d)}=c_{\alpha \beta}^{\gamma}\ X_{\gamma},
\]
and
\[
\Omega_{ij}^{\alpha},\ i,j=1,\dots,d,\ \alpha=1,\dots,\frac{d(d-1)}{2}
\]
are the components of the curvature form $\Omega^{\theta}$ defined through
\begin{align*}
\Omega^{\theta}=&d\theta +\frac{1}{2}[\theta,\theta]\\
=&\frac{1}{2}\sum_{i,j,\alpha}\Omega^{\alpha}_{i j}\ X_{\alpha}^{\sharp}\ \mathfrak{s}^i\wedge \mathfrak{s}^j\ (\Omega^{\alpha}_{ij}=-\Omega^{\alpha}_{ji}).
\end{align*}
Since the standard inner product of $\mathfrak{o}(d)$ is $O(d)$-invariant, $\{c_{\alpha \beta}^{\gamma} \}$ is totally anti-symmetric in $\alpha, \beta, \gamma$. The following \psref{covariant} and \ref{integralcurve} can be easily obtained by \eqref{co}.
\begin{prop}\label{covariant}
For any $u\in \mathcal{O}(M)$ and $A\in T_u\mathcal{O}(M)$, it holds that
\begin{gather*}
\tilde \nabla \theta (A,A)=0.
\end{gather*}
Furthermore, if $A$ is horizontal, then
\begin{gather*}
\tilde \nabla \mathfrak{s}(A,A)=0.
\end{gather*}
\end{prop}
\begin{proof}
Any tangent vector $A$ can be denoted by
\[
A=a^iL_i(u)+b^{\alpha}X_{\alpha}^{\sharp},\ a^i,b^{\alpha}\in \mathbb{R},\ i=1,\dots,d,\ \alpha=1,\dots,\frac{d(d-1)}{2}.
\]
Therefore we can write
\begin{align*}
\tilde \nabla \theta^{\alpha} (A,A)=&a^ka^l(\tilde \nabla \theta^{\alpha})(L_k,L_l)+a^kb^{\gamma}\tilde \nabla \theta^{\alpha}(L_k,X_{\gamma}^{\sharp})\\
&+b^{\beta}a^l\tilde \nabla \theta^{\alpha}(X_{\beta}^{\sharp},L_l)+b^{\beta}b^{\gamma}(\tilde \nabla \theta^{\alpha})(X_{\beta}^{\sharp},X_{\gamma}^{\sharp}).
\end{align*}
By using \eqref{co}, it holds that
\begin{align*}
&a^ka^l(\tilde \nabla \theta^{\alpha})(L_k,L_l)=a^ka^l\langle \omega^{\alpha}_l,L_k \rangle
= \frac{a^ka^l}{2} \Omega_{l k}^{\alpha},\\
&\tilde \nabla \theta^{\alpha}(L_k,X_{\gamma}^{\sharp})=-\langle \omega_{\gamma}^{\alpha},L_k\rangle=0,\\
&\tilde \nabla \theta^{\alpha}(X_{\beta}^{\sharp},L_l)=-\langle \omega_l^{\alpha},X_{\beta}^{\sharp}\rangle=0,\\
&b^{\beta}b^{\gamma}(\tilde \nabla \theta^{\alpha})(X_{\beta}^{\sharp},X_{\gamma}^{\sharp})=b^{\beta}b^{\gamma}\langle \omega_{\gamma}^{\alpha},X_{\beta}^{\sharp} \rangle
=-b^{\beta}b^{\gamma}c^{\alpha}_{\beta \gamma}.
\end{align*}
Since $\Omega _{lk}^{\alpha}=-\Omega_{kl}^{\alpha}$, we obtain
\[
a^ka^l(\tilde \nabla \theta^{\alpha})(L_k,L_l)=0.
\]
Similarly, we obtain
\[
b^{\beta}b^{\gamma}(\tilde \nabla \theta^{\alpha})(X_{\beta}^{\sharp},X_{\gamma}^{\sharp})=0.
\]
Therefore we deduce $\tilde \nabla \theta (A,A)=0.$ Next suppose $A$ is horizontal. Then we obtain
\begin{align*}
\tilde \nabla \mathfrak{s}^j(A,A)=&a^ka^l \tilde \nabla \mathfrak{s}^j (L_k,L_l)
=-a^ka^l\langle \omega_l^j\  L_k \rangle
=0.
\end{align*}
This proves the proposition.
\end{proof}
\begin{prop}\label{integralcurve}
\begin{itemize}
\item[(1)]Integral curves of the horizontal vector field $a^kL_k$, $a_k\in \mathbb{R}^k,\ k=1,\dots,d$, are geodesics with respect to $\tilde \nabla$. 
\item[(2)]Integral curves of the vertical vector field $b^{\alpha}X_{\alpha}^{\sharp}$, $b^{\alpha}\in \mathbb{R}^d,\ \alpha=1,\dots,\frac{d(d-1)}{2}$, are geodesics with respect to $\tilde \nabla$.
\end{itemize}
\end{prop}
\begin{proof}
Let $u(t)$ be a curve on $\mathcal{O}(M)$ satisfying
\[
\frac{du}{dt}(t)=a^kL_k(u(t)).
\]
Then it holds that
\begin{align*}
\tilde \nabla_{\frac{d}{dt}}\frac{du}{dt}={}& a^ka^l(\tilde \nabla_{L_l}L_k)(u(t))\\
={}&a^ka^l(\omega_{kl}^{\alpha}X_{\alpha}^{\sharp})\\
={}&-\frac{1}{2}(a^ka^l\Omega_{k l}^{\alpha})X_{\alpha}^{\sharp}\\
={}&0.
\end{align*}
Next let us denote the integral curve of $b^{\alpha}X_{\alpha}^{\sharp}$ by $v(t)$. Then
\begin{align*}
\tilde \nabla_{\frac{d}{dt}}\frac{dv}{dt}={}&b^{\beta}b^{\gamma}(\tilde \nabla_{X_{\beta}^{\sharp}}X_{\gamma}^{\sharp})(v(t))\\
={}&\frac{b^{\beta}b^{\gamma}}{2}c_{\beta \gamma}^{\lambda}\ X_{\lambda}^{\sharp}(v(t))\\
={}&0.
\end{align*}
This completes the proof.
\end{proof}
The next proposition proved in \cite{PE} expresses the relation between geodesics on $\mathcal{O}(M)$ and those on $M$.
\begin{prop}[\cite{PE}, Proposition 1.9]\label{geodesic}
\begin{itemize}
\item[(1)]Let $x$ and $y$ be two points in $M$ and $c$ a minimal geodesic from $x$ to $y$. Let us denote the parallel transport along $c$ by $P_c:T_xM\to T_yM$. Let $u\in \mathcal{O}_x(M)$, $v\in \mathcal{O}_y(M)$ with $v=P_c\circ u$. Then minimal geodesics from $u$ to $v$ with respect to $\tilde g$ are horizontal and the horizontal lift of $c$ starting at $u$ is one of minimal geodesics from $u$ to $v$. Furthermore, if minimal geodesics from $x$ to $y$ on $M$ are unique, then minimal geodesics from $u$ to $v$ are also unique.
\item[(2)] Let $\tau$ be a geodesic on $\mathcal{O}(M)$ with respect to $\tilde g$. Suppose that $\tau'(0)$ is horizontal. Then $\tau$ is a horizontal curve and $\pi \circ \tau$ is a geodesic on $M$.
\end{itemize}
\end{prop}
We prepare a lemma and propositions for later use in \sref{sectionmain}.
\begin{lem}\label{innerprod}
Let $u\in \mathcal{O}(M)$, $A,B\in T_u\mathcal{O}(M)$ and $a\in O(d)$. Then
\[
\tilde{g}(R_{a*}A,R_{a*}B)=\tilde{g}(A,B).
\]
\end{lem}
\begin{proof}
By definition, it holds that
\begin{align*}
\tilde{g}(R_{a*}A,R_{a*}B)&=\langle \langle \theta (ua),R_{a*}A\rangle, \langle \theta (ua),R_{a*}B\rangle \rangle_{\mathfrak{o}(d)}\\
&\h +\langle \langle \mathfrak{s}(ua),R_{a*}A\rangle,\langle \mathfrak{s}(ua),R_{a*}B\rangle \rangle_{\mathbb{R}^d}.
\end{align*}
Since $\langle \theta (ua),R_{a*}A\rangle=Ad(a^{-1})\langle \theta (u),A\rangle$ and the metric on $\mathfrak{o}(d)$ is $Ad$-invariant,
\[
\langle \langle \theta (ua),R_{a*}A\rangle, \langle \theta (ua),R_{a*}B\rangle \rangle_{\mathfrak{o}(d)}=\langle \langle\theta (u),A\rangle,\langle \theta (u),B\rangle \rangle_{\mathfrak{o}(d)}.
\]
Furthermore,
\begin{align*}
\langle \mathfrak{s}(ua),R_{a*}A\rangle&=(ua)^{-1}\pi_*R_{a*}A\\
&=a^{-1}\circ u^{-1}(\pi \circ R_a)_*A\\
&=a^{-1}\circ u^{-1}(\pi _*A)\\
&=a^{-1}\langle \mathfrak{s}(u),A\rangle.
\end{align*}
Since $a$ is isometric,
\[
\langle \langle \mathfrak{s}(ua),R_{a*}A\rangle ,\langle \mathfrak{s}(ua),R_{a*}B\rangle \rangle_{\mathbb{R}^d}=\langle \langle \mathfrak{s}(u),A\rangle,\langle \mathfrak{s}(u),B\rangle \rangle_{\mathbb{R}^d}.
\]
Therefore $\tilde{g}(R_{a*}A,R_{a*}B)=\tilde{g}(A,B).$
\end{proof}
\begin{prop}\label{minimal}
Let $u,v\in \mathcal{O}(M)$. Then for all $a\in O(d)$,
\[
d_{\mathcal{O}(M)}(ua,va)= d_{\mathcal{O}(M)}(u,v),
\]
where $d_{\mathcal{O}(M)}$ is the Riemannian distance with respect to $\tilde g$.
\end{prop}
\begin{proof}
For any $\varepsilon>0$, there exists a curve $\tau_{\varepsilon}:[0,1]\to \mathcal{O}(M)$ with $\tau_{\epsilon}(0)=u$, $\tau_{\epsilon}(1)=v$ satisfying $\displaystyle \left| \frac{d\tau_{\varepsilon}}{dt} \right| \leq d_{\mathcal{O}(M)}(u,v)+\varepsilon$. Then by \lref{innerprod},
\begin{align*}
\int_0^1\left| \frac{d}{dt}R_a\tau_{\varepsilon} \right| dt&=\int_0^1 \left| R_{a*}\frac{d\tau_{\varepsilon}}{dt}\right| dt\\
&=\int_0^1\left| \frac{d\tau_{\varepsilon}}{dt} \right| dt\\
&\le d_{\mathcal{O}(M)}(u,v)+\varepsilon.
\end{align*}
Thus
\[
d_{\mathcal{O}(M)}(ua,va)\leq d_{\mathcal{O}(M)}(u,v)+\varepsilon.
\]
Since $\varepsilon$ is arbitrary,
\[
d_{\mathcal{O}(M)}(ua,va)\leq d_{\mathcal{O}(M)}(u,v).
\]
This inequality holds for all $u,v\in \mathcal{O}$ and $a\in O(d)$. Therefore
\[
d_{\mathcal{O}(M)}(ua,va)= d_{\mathcal{O}(M)}(u,v),
\]
and the proposition follows.
\end{proof}

\begin{prop}\label{geo}
Let $u,v\in \mathcal{O}(M)$ and $\tau (t)\ (t\in [0,1])$ a minimal geodesic from $u$ to $v$. Then for all $a\in O(d)$, $R_a\tau$ is a minimal geodesic from $ua$ to $va$.
\end{prop}
\begin{proof}
By \lref{innerprod} and \pref{minimal}, we obtain
\begin{align*}
\int_0^1\left| \frac{d}{dt}R_a\tau \right| dt&=\int \left| R_{a*}\frac{d\tau}{dt} \right| dt\\
&=\int_0^1\left| \frac{d\tau}{dt} \right| dt\\
&=d_{\mathcal{O}(M)}(u,v)\\
&=d_{\mathcal{O}(M)}(ua,va).
\end{align*}
Therefore $R_a\tau$ is a minimal geodesic.
\end{proof}

\end{document}